\documentclass[12pt]{article}
\usepackage{graphicx}
\usepackage{amsmath}
\usepackage{amssymb}
\usepackage{theorem}

   \usepackage{hyperref}

\numberwithin{equation}{section}

\newtheorem{thm}{Theorem}[section]
\newtheorem{lemma}[thm]{Lemma}
\newtheorem{prop}[thm]{Proposition}
\newtheorem{cor}[thm]{Corollary}
{\theorembodyfont{\rmfamily}
\newtheorem{defn}[thm]{Definition}
\newtheorem{example}[thm]{Example}

\newtheorem{rmk}[thm]{Remark}
}

\newcommand{\qed}{\hfill \mbox{\raggedright \rule{.07in}{.1in}}}
 
\newenvironment{proof}{\vspace{1ex}\noindent{\bf
Proof}\hspace{0.5em}}{\hfill\qed\vspace{1ex}}

\newcommand{\R}{{\mathbb R}}
\newcommand{\T}{{\mathbb T}}
\newcommand{\Z}{{\mathbb Z}}
\newcommand{\N}{{\mathbb N}}
\newcommand{\bbA}{{\mathbb A}}
\newcommand{\bC}{{\overline{\mathbb C}}}

\newcommand{\cA}{\mathcal{A}}

\newcommand{\cH}{\mathcal{H}}

\newcommand{\cP}{\mathcal{P}}

\newcommand{\eps}{\epsilon}
\newcommand{\Leb}{\operatorname{Leb}}
\newcommand{\BV}{\operatorname{BV}}
\newcommand{\diam}{\operatorname{diam}}
\newcommand{\dist}{\operatorname{dist}}

\title{Global-local mixing for \\ infinite measure dynamical systems}

\author{
  Douglas Coates
  \thanks{Instituto de Matemáticas, Universidade Federal do Rio de Janeiro, Rio de Janeiro, Brazil}
  \and
  Ian Melbourne
  \thanks{Mathematics Institute, University of Warwick, Coventry CV4 7AL, United Kingdom}
}

\date{10 May 2025, updated 20 December 2025}

\begin{document}

\maketitle

 \begin{abstract}
We prove global-local mixing for a large class of dynamical systems with infinite invariant measure. In particular, we treat
intermittent maps including maps with multiple neutral fixed points, nonMarkovian intermittent maps, and multidimensional nonMarkovian intermittent maps.
We also prove global-local mixing for parabolic rational maps of the complex plane.
 \end{abstract}

\section{Introduction}

There has been a significant surge of interest recently in the ergodic theory of dynamical system preserving an infinite $\sigma$-finite measure~\cite{Aaronson}.
A remark in~\cite[p.~75]{Aaronson} suggests that there is no reasonable notion of mixing for infinite measure systems. Nevertheless, there are at least two notions of mixing that have proved fruitful: \emph{Krickeberg mixing}~\cite{Krickeberg67} and more recently \emph{global-local mixing} introduced by Lenci~\cite{Lenci10}.

Let $(X,\mu)$ be an infinite, $\sigma$-finite measure space.  We suppose that 
$f:X\to X$ is a conservative, exact, measure-preserving transformation.
The usual mixing property $\lim_{n\to\infty}\mu(A\cap f^{-n}B)=\mu(A)\mu(B)$ for all (finite measure) measurable sets $A$, $B$ fails.
Moreover, by~\cite{HajianKakutani64} there always exists a measurable set $W$
with $0<\mu(W)<\infty$ such that
$\mu(W\cap f^{-n}W)=0$ infinitely often.
The system is Krickeberg mixing if there exist constants $a_n\to\infty$ such that 
\[
\lim_{n\to\infty}a_n\int_X v\;\,w\circ f^n\,d\mu=\int_X v\,d\mu\int_X w\,d\mu
\]
for a sufficiently large class of observables $v$ and $w$.
The more recent notion of global-local mixing~\cite{Lenci10} asks that
\begin{equation} \label{eq:glocal}
\lim_{n\to\infty}\int_X v\;\,g\circ f^n\,d\mu=\bar g\,\int_X v\,d\mu
\end{equation}
for all $v\in L^1$ and a sufficiently large class of ``global'' observables $g$.

Roughly speaking, an observable $g\in L^\infty$ is \emph{global} if
there exists $\bar g\in\R$ such that
\[
\lim_{m\to\infty}\frac{\int_{Z_m} g\,d\mu}{\mu(Z_m)}= \bar g
\]
for suitable nested sequences of measurable subsets $Z_m\subset X$
	with $\mu(Z_m)<\infty$ and $\mu(X\setminus \bigcup_m Z_m)=0$.
The precise notion of global observable depends on the context and is made precise for our purposes in Definition~\ref{def:global} below.

Prototypical examples are the intermittent maps introduced in~\cite{PomeauManneville80}. Krickeberg mixing was studied under restrictive conditions in~\cite{Thaler00} and a full theory was developed by~\cite{Gouezel11,MT12,Terhesiu15,Terhesiu16}.
Likewise, global-local mixing for intermittent maps was studied under restrictive conditions in~\cite{BonannoGiuliettiLenci18,BonannoLenci21,CanestrariLenci_arxiv} and a full theory is developed in the current paper. Results on global-local mixing for other classes of dynamical systems can be found in~\cite{BonannoGiuliettiLenci18b,DolgopyatNandori22,GiuliettiHammerlindlRavotti22}.

Bonanno \& Lenci~\cite{BonannoLenci21} proved global-local mixing for certain intermittent maps $f:X\to X$ with $X=[0,1]$, including generalised classical Pomeau-Manneville maps of the form
\[
fx=x+b x^{\alpha+1}\quad\bmod 1
\]
where $\alpha\in(0,1]$ and $b\in\Z^+$, as well as generalised Liverani-Saussol-Vaienti (LSV) maps~\cite{LSV99}. The latter have finitely or infinitely many full branches where the first branch is of the form $x+b x^{\alpha+1}$ with $b>0$ and the remaining branches are linear with positive slope. They also studied perturbations of these two classes of maps. Full results were obtained for $\alpha\in(0,1)$ and a weaker result was obtained for $\alpha=1$.

One aim of the current paper is to remove various restrictions in~\cite{BonannoLenci21} thereby developing a complete theory of global-local mixing for intermittent maps:
\begin{itemize}

\parskip=-1pt
\item The examples mentioned so far have a single neutral fixed point and the previous techniques do not apply when there is more than one neutral fixed points. We allow any number of neutral fixed points. 
\item We remove the assumption that $f$ is full-branch or even Markov.
\item Full global-local mixing is proved even in the case $\alpha=1$.
\item There is a technical assumption (A5) in~\cite{BonannoLenci21} which is hard to verify and need not hold, but is unnecessary for our approach.
\end{itemize}

It is worth noting that the initial results of~\cite{Thaler00} on Krickeberg mixing for intermittent maps also required a technical assumption that was hard to verify and not always true. The approach in~\cite{MT12} eliminated the need for such an assumption, and a key step in the current paper is to apply the result of~\cite{MT12}.

\begin{rmk} For certain classes of dynamical systems, global-local mixing seems to be a more flexible property than Krickeberg mixing. Certainly, there are numerous examples in~\cite{DolgopyatNandori22} which are shown to be global-local mixing and where proving Krickeberg mixing is way beyond reach.
In contrast, a full theory of global-local mixing for intermittent maps seems to be more delicate than Krickeberg mixing: Krickeberg mixing is just one step in our approach (see condition~\eqref{eq:K} below) and further arguments are required to achieve global-local mixing.
\end{rmk}

The remainder of this paper is organised as follows.
In Section~\ref{sec:main}, we state and prove our main result, Theorem~\ref{thm:main}, in a suitable abstract setting.
In Section~\ref{sec:proof}, we prove Theorem~\ref{thm:main} by dividing into piecewise constant and piecewise mean zero observables.

In Section~\ref{sec:ex}, we apply our result to large classes of one-dimensional intermittent maps, including those in~\cite{BonannoLenci21}, but also maps with multiple neutral fixed points and nonMarkovian intermittent maps.
In addition, we treat a multidimensional nonMarkovian intermittent map studied in~\cite{EMV21} as well as parabolic rational maps of the complex plane.

\vspace{-2ex}
\paragraph{Notation}
We use ``big O'' and $\ll$ notation interchangeably, writing $a_n=O(b_n)$ or $a_n\ll b_n$
if there are constants $C>0$, $n_0\ge1$ such that
$a_n\le Cb_n$ for all $n\ge n_0$.
We write $a_n\approx b_n$ if $a_n\ll b_n$ and $b_n\ll a_n$.
As usual, $a_n=o(b_n)$ means that $a_n/b_n\to0$ and
$a_n\sim b_n$ means that $a_n/b_n\to1$.

\section{The abstract setup}
\label{sec:main}

Let $(X,\mu)$ be an infinite, $\sigma$-finite measure space.  We suppose that 
$f:X\to X$ is a conservative, exact, measure-preserving transformation.

Fix a subset $Y\subset X$ with $0<\mu(Y)<\infty$. 
By conservativity, the first hit map
\[
\tau:X\to\Z^+, \qquad \tau(x)=\inf\{n\ge1:f^nx\in Y\}
\]
is finite almost everywhere.
Set
\[
Y_n=\{y\in Y: \tau(y)=n\}, 
\qquad
X_n=\{x\in X\setminus Y:\tau(x)=n\}.
\]

\begin{prop} \label{prop:X0}
$\mu(X_n)<\infty$ for all $n\ge1$, and $\lim_{n\to\infty}\mu(X_n)=0$.
\end{prop}

\begin{proof}
Note that $\mu(X_n)\le \mu(f^{-n}Y)=\mu(Y)<\infty$ by $f$-invariance of $\mu$.
Also,
\[
\sum_{j=1}^n \mu(X_j)=
\sum_{j=1}^n \mu(f^{-1}X_j)
=\sum_{j=1}^n \mu(X_{j+1})
+\sum_{j=1}^n \mu(Y_{j+1}).
\]
Hence 
\[
\mu(X_1)-\mu(X_{n+1})=\sum_{j=2}^{n+1}\mu(Y_j)\to \mu(Y)-\mu(Y_1).
\]
Similarly, $\mu(Y)=\mu(f^{-1}Y)=\mu(Y_1)+\mu(X_1)$ 
and it follows that $\mu(X_n)\to0$.
\end{proof}

We can now give the precise definitions of global observable and global-local mixing:
\begin{defn} \label{def:global}
We say that $g\in L^\infty$ is a \emph{global observable}
	if there exists $\bar g\in\R$ such that\footnote{
By Proposition~\ref{prop:X0}, this coincides with the rough definition in the Introduction, taking $Z_m=X_m\cup Y$.}
	\[
		\lim_{n\to\infty}\frac{\int_{\bigcup_{i=1}^n X_i}g\,d\mu}
		{\mu(\bigcup_{i=1}^n X_i)}=\bar g.
	\]
The global observable is \emph{centred} if $\bar g=0$.

The dynamical system $f:X\to X$ is \emph{(fully) global-local mixing} if~\eqref{eq:glocal} holds for all $v\in L^1$ and global observables $g\in L^\infty$.
\end{defn}

Note that $f^jY_{j+n}\subset X_n$ for all 
$j,n\ge1$. 
In general $f^j|_{Y_{j+n}}:Y_{j+n}\to X_n$ need not be a bijection.
We assume that there exist integers $n_0,\,q\ge1$, an indexing set $\bbA$ (finite or countably infinite), a surjection
$\psi:\bbA\to\{1\le p\le q\}$,
 and decompositions
\[
X_n=\bigcup_{p=1}^qX_{n,p}, \qquad
	Y_n=\bigcup_{p=1}^q\bigcup_{\psi(r)=p}Y_{n,r}, 
\quad n\ge n_0,
\] 
into disjoint sets
such that $f^j$ maps each $Y_{j+n,r}$ (for $n\ge n_0$) bijectively onto
$X_{n,\psi(r)}$.

\begin{rmk} \label{rmk:p}
For intermittent maps, $q$ represents the number
  of neutral fixed points and the $X_{ n,p }$ are the preimages of $Y$
  under $f^n$ that accumulate at the $p$'th fixed point $\xi_p$. The indexing set
  $\mathbb{A}$ enumerates for $1\le p\le q$ the branches of $f$ which contain $\xi_p$ in their range and not their domain, and
  $\{Y_{ n,r },\, r\in\psi^{-1}(p)\}$ is the set of preimages of $X_{n - 1,p}$ in the corresponding branches. 

In particular, for
intermittent maps with one neutral fixed point and two branches (such as the LSV map~\cite{LSV99}), we can take $q=|\bbA|=1$, see Example~\ref{ex:2branch}.

The abstract setting can be generalised in various ways. For instance, in Example~\ref{ex:EMV} we consider bijections $f^j:Y_{j,n,r}\to X_{n,p}$ where the decompositions $\bigcup_{\psi_j(r)=p}Y_{j,n,r}$ depend on $j$ and $n$ (instead of just $j+n$) with surjections $\psi_j:\bbA_j\to\{1,\dots,q\}$ depending on $j$. We prefer to emphasise the current setting since it suffices for all of the remaining examples with less onerous notation.
\end{rmk}

We assume that there are constants $\alpha\in(0,1]$
and $\gamma_p>0$, $1\le p\le q$, 
 such that
\begin{equation} \label{eq:Y}
\sum_{\psi(r)=p}\sum_{j\ge n}\mu(Y_{j,r}) \sim \gamma_p \, n^{-\alpha}
\quad\text{as $n\to\infty$ for all $1\le p\le q$.}
\end{equation}
In particular,
\begin{equation} \label{eq:tau}
\mu(y\in Y:\tau(y)\ge n) \sim \gamma n^{-\alpha}
\quad\text{as $n\to\infty$}
\end{equation}
where $\gamma=\sum_{p=1}^q\gamma_p\in(0,\infty)$.

\begin{rmk}
When $q=1$, conditions~\eqref{eq:Y} and~\eqref{eq:tau} coincide.
\end{rmk}

Our main assumption is a Krickeberg mixing condition on $f$.
Let $L:L^1(X,\mu)\to L^1(X,\mu)$ denote the transfer operator of $f$
(so $\int_X Lv\,w\,d\mu=\int_X v\,w\circ f\,d\mu$ for
$v\in L^1(X,\mu)$, $w\in L^\infty$).
We require that there is a constant $K\ge0$ such that
\begin{equation} \label{eq:K}
	\lim_{n\to\infty}a_n\big|1_Y(L^n1_Y-Ka_n^{-1})\big|_\infty =0.
\end{equation}

\begin{rmk} \label{rmk:K}
Condition~\eqref{eq:K} has been verified by~\cite{Gouezel11,MT12} in many examples (with $K>0$).
In particular, 
when $F=f^\tau:Y\to Y$ is a mixing Gibbs-Markov map or a mixing AFU map (see Section~\ref{sec:AFN}), condition~\eqref{eq:K} holds under~\eqref{eq:tau}
provided $\alpha\in(\frac12,1]$.
When $\alpha\in(0,\frac12]$, condition~\eqref{eq:K} holds if in addition $\mu(y\in Y:\tau(y)=n)=O(n^{-(\alpha+1)})$.
\end{rmk}

Let $g\in L^\infty$. We fix $1\le p\le q$ and 
note that 
\[
\sum_{\psi(r)=p}\int_{Y_{j+i,r}}g\circ f^j\,d\mu=
\sum_{\psi(r)=p}\int_{X_{i,p}}g\,d(f^j|_{Y_{j+i,r}})_*\mu=
\int_{X_{i,p}}gJ_{j,i,p}\,d\mu
\]
where 
\[
J_{j,i,p}: X_{i,p}\to(0,\infty), \qquad 
J_{j,i,p} =\sum_{\psi(r)=p}\frac{d(f^j|_{Y_{j+i,r}})_*\mu}{d\mu|_{X_{i,p}}}.
\]
Our final assumption is that there are constants $c_{j,i,p}>0$ such that
\begin{equation} \label{eq:J}
\lim_{j\to\infty}\sum_{\eps j < i < \eps^{-1}j} \sup_{X_{i,p}}|J_{j,i,p}-c_{j,i,p}|=0
\quad\text{for each $p=1,\dots,q$ and all $\eps>0$}.
\end{equation}

\begin{thm} \label{thm:main}
Assume conditions~\eqref{eq:Y},~\eqref{eq:K} and~\eqref{eq:J}.
Then $f$ is global-local mixing.
\end{thm}

\section{Proof of Theorem~\ref{thm:main}}
\label{sec:proof}

This section is devoted to the proof of Theorem~\ref{thm:main}.
We continue to assume that
$(X,\mu)$ is an infinite, $\sigma$-finite measure space and that 
$f:X\to X$ is a conservative, exact, measure-preserving transformation.
In proving Theorem~\ref{thm:main}, it can be supposed without loss of generality that $g$ is centred.
We let $\mu_Y=\mu|_Y$.

In Subsection~\ref{sec:global}, we give a useful characterisation of global observable.
In Subsection~\ref{sec:exact}, we mention some simplifications arising from exactness.
In Subsection~\ref{sec:K}, we show how to use the Krickeberg mixing assumption~\eqref{eq:K} to reduce the global-local mixing problem. This leads naturally to the splitting of global observables into those that are either piecewise constant or piecewise mean zero.

In Subsection~\ref{sec:pwc}, we state and prove Lemmas~\ref{lem:pwc} and~\ref{lem:pwc1}: global-local mixing for piecewise constant global observables.
In Subsection~\ref{sec:pw0}, we state and prove Lemma~\ref{lem:pw0}: global-local mixing for piecewise mean zero observables
(it is only here that the Jacobian condition~\eqref{eq:J} is used).

Theorem~\ref{thm:main} is an immediate consequence of Lemmas~\ref{lem:pwc},~\ref{lem:pwc1} and~\ref{lem:pw0}.

\subsection{Characterisation of global observables}
\label{sec:global}

\begin{prop} \label{prop:X}
Assume condition~\eqref{eq:Y}. Then
$\mu(X_{n,p})\sim \gamma_p \, n^{-\alpha}$ as $n\to\infty$ for all $1\le p\le q$.
\end{prop}

\begin{proof}
By $f$-invariance of $\mu$,
\[
\mu(X_{n-1,p})=
\mu(f^{-1}X_{n-1,p})=
\mu(X_{n,p})+\sum_{\psi(r)=p}\mu(Y_{n,r}).
\]
 By Proposition~\ref{prop:X0}, $\mu(X_{k,p})\to0$ as $k\to\infty$.
Hence
\begin{align*}
\mu(X_{n,p}) & =\lim_{k\to\infty}\big(\mu(X_{n,p})- \mu(X_{k,p})\big)
=\sum_{j\ge n+1} \big(\mu(X_{j-1,p})- \mu(X_{j,p})\big)
\\ & =
\sum_{\psi(r)=p}
\sum_{j\ge n+1} 
\mu(Y_{j,r})
\sim \gamma_pn^{-\alpha}
\end{align*}
by~\eqref{eq:Y}. 
\end{proof}

Define positive sequences $a_n$, $n\ge0$, by setting
\[
a_n=\begin{cases} n^{1-\alpha}, & 0<\alpha<1 \\
\log n, & \alpha=1
\end{cases} \;.
\]
Whenever $a_n^{-1}$ is undefined, we set $a_n=1$.

Define sequences
\[
\bar g_{n,p}= \frac{\int_{X_{n,p}}g\,d\mu}{\mu(X_{n,p})}, \quad n\ge n_0,
\qquad\text{and}\qquad
\bar g_n=\sum_{p=1}^q\gamma_p\bar g_{n,p}.
\]
The sequences are bounded since $g\in L^\infty$; indeed 
$|\bar g_{n,p}|\le |g|_\infty$ 
and $|\bar g_n|\le \gamma |g|_\infty$.

\begin{lemma} \label{lem:g}
Assume condition~\eqref{eq:Y} holds and let 
$g\in L^\infty$. Then  $g$ is a centred global observable if and only if
\begin{equation} \label{eq:g}
	\lim_{n\to\infty} a_n^{-1}\sum_{i=n_0}^n i^{-\alpha}\bar g_i=0.
\end{equation}
\end{lemma}

\begin{proof}
	By Proposition~\ref{prop:X}, $\mu(X_n)\sim \gamma n^{-\alpha}$ so there exists $\tilde\gamma>0$ such that
	$\mu\big(\bigcup_{i=n_0}^n X_i\big)\sim \tilde\gamma a_n$.
On the other hand, by definition of $\bar g_{i,p}$, $\bar g_i$ and Proposition~\ref{prop:X},
\begin{align*}
\int_{\bigcup_{i=n_0}^n X_i} g\,d\mu  & =
		\sum_{i=n_0}^n\sum_{p=1}^q\int_{X_{i,p}} g\,d\mu
		 =\sum_{i=n_0}^n\sum_{p=1}^q \bar g_{i,p}\,\mu(X_{i,p})
\\ &
 =\sum_{i=n_0}^n\sum_{p=1}^q \bar g_{i,p}\,\gamma_p\,i^{-\alpha}
+o(a_n)=
 \sum_{i=n_0}^n \bar g_i\,i^{-\alpha} +o(a_n).
\end{align*}
Hence, $g$ is a centred global observable if and only if~\eqref{eq:g} holds.
\end{proof}

\subsection{Consequences of exactness}
\label{sec:exact}

By exactness of $f$, it suffices
(see~\cite[Lemma~3.3(a)]{BonannoLenci21}) to prove that~\eqref{eq:glocal} holds for
one $v\in L^1(X,\mu)$ with $\int_X v\,d\mu\neq0$ in order to conclude 
global-local mixing.
Choosing $v=1_Y$, we reduce to showing that
\[
\lim_{n\to\infty} \int_Y g\circ f^n\,d\mu = 0
\]
for all centred global observables $g\in L^\infty$.

The following result sometimes leads to simplified combinatorics, see Example~\ref{ex:parabolic}.

\begin{prop} \label{prop:power}
If $f^m:X\to X$ is global-local mixing for some $m\ge1$, then $f$ is global-local mixing.
\end{prop}

\begin{proof}
Fix a centred global observable $g\in L^\infty$.
We claim that $g_\ell=g\circ f^\ell$ is 
a centred global observable for each $\ell\ge0$.
Then, writing $n=km+\ell$ where $k\in\N$ and $0\le \ell\le m-1$,
\[
\int_Y g\circ f^n\,d\mu
=\int_Y g_\ell\circ f^{km}\,d\mu
\qquad\text{and}\qquad
\lim_{k\to\infty}\int_Y g_\ell\circ f^{km}\,d\mu=0
\]
since $f^m$ is global-local mixing and $g_\ell$ is centred global.
Hence $f$ is global-local mixing.

It remains to prove the claim.
Notice that for
$i \ge \ell + 1$, we can write
$f^{-\ell} X_{i - \ell}$ as a disjoint union
$f^{-\ell} X_{i - \ell} = X_{i} \cup Z_{i,\ell}$ where
$Z_{i,\ell} \subset \bigcup_{ k = 0 }^{ \ell - 1 } f^{-k}Y$.
Hence, for $n \ge \ell + 1$
\[
\int_{\bigcup_{ i = 1 }^{n} X_i }  g_{\ell} \, d\mu
= \int_{\bigcup_{ i = 1}^{ \ell } X_i} g_{\ell} \, d\mu
+  \int(1_{\bigcup_{ i = \ell + 1}^{ n }X_{i - \ell}} g)\circ f^\ell \, d\mu
-  \int_{\bigcup_{ i = \ell + 1}^{ n }Z_{i,\ell}} g_{\ell} \, d\mu.
\]
Moreover,
\[
\left|\int_{\bigcup_{ i = \ell + 1}^{ n }Z_{i,\ell}} g_{\ell} \, d\mu \right|
\le
| g |_{\infty}
\mu\left( \bigcup_{ k = 0}^{ \ell - 1 } f^{-k} Y \right)
\le \ell | g |_{\infty} \mu (Y).
\]
It follows that
\begin{align*}
\int_{\bigcup_{i=1}^n X_i}g_\ell\,d\mu
= \int_{\bigcup_{i = 1}^{n - \ell }X_i} g \,d\mu+O(1)
 = \int_{\bigcup_{i=1}^n X_i}g\,d\mu+O(1).
\end{align*}
Dividing by $\mu(\bigcup_{i=1}^n X_i)$, we obtain that $g_\ell$ is a centred global observable as claimed.~
\end{proof}

\subsection{Consequences of Krickeberg mixing}
\label{sec:K}

We now show how to use condition~\eqref{eq:K} to reduce the problem of global-local mixing.

\begin{prop} \label{prop:CMT}
	The sequence 
$\sum_{j=1}^n a_{n-j}^{-1}\,j^{-\alpha}$ is convergent, with limit $\pi(\sin\alpha\pi)^{-1}$ for $\alpha\in(0,1)$ and $1$ for $\alpha=1$.
In addition, if $\lim_{n\to\infty}\delta(n)=0$, then
\[
\lim_{n\to\infty}\sum_{j=1}^n a_{n-j}^{-1}\,j^{-\alpha}\delta(j)=
\lim_{n\to\infty}\sum_{j=1}^n a_{n-j}^{-1}\,j^{-\alpha}\delta(n-j)=0.
\]
\end{prop}

\begin{proof}
For the first statement, see for example~\cite[Lemmas~A.1 and~A.4]{CMTsub}.
The second statement holds by~\cite[Lemma~A.2]{CMTsub} for $0<\alpha<1$ and
by~\cite[Lemma~A.4]{CMTsub} for $\alpha=1$.~
\end{proof}

\begin{lemma}  \label{lem:K}
Assume conditions~\eqref{eq:tau} and~\eqref{eq:K} and let 
$g\in L^\infty$. Then $\lim_{n\to\infty} \int_Y g\circ f^n\,d\mu=0$ if and only if 
\[
\lim_{n\to\infty}\sum_{1\le j\le n}
a_{n-j}^{-1} \sum_{\eps j<i< \eps^{-1}j}\int_{Y_{j+i}}g\circ f^j\,d\mu=0
\qquad\text{for all $\eps>0$}.
\]
In particular, it suffices that
\[
\lim_{j\to\infty}\; j^\alpha \!\!\!\!\!\!
\sum_{\eps j<i< \eps^{-1}j}\int_{Y_{j+i}}g\circ f^j\,d\mu=0
\qquad\text{for all $\eps>0$}.
\]
\end{lemma}

\begin{proof}
Write $Y$ as a disjoint union
\[
Y=\bigcup_{j=0}^n\big\{y\in Y:f^{n-j}y\in Y,\,\tau(f^{n-j}y)>j\big\}.
\]
Then
\begin{align*}
\int_Y g\circ f^n\,d\mu & =\sum_{j=0}^n\int_X 1_Y  1_Y \circ f^{n-j} 1_{\{\tau>j\}}\circ f^{n-j} g\circ f^n\,d\mu \\
	&  =\sum_{j=0}^n \int_X 1_Y (L^{n-j}1_Y)\, 1_{\{\tau>j\}}\, g\circ f^j\,d\mu.
\end{align*}
	By~\eqref{eq:K},
\[
 \int_Y g\circ f^n\,d\mu 
	 = K\sum_{j=0}^n a_{n-j}^{-1}\int_Y 1_{\{\tau>j\}}g\circ f^j\,d\mu + O\Big(|g|_\infty \sum_{j=0}^n e_{n-j} \, \mu(y\in Y:\tau(y)>j)\Big)
\]
	where $e_n=o(a_n^{-1})$ and $\mu_Y(\tau>j)=O(j^{-\alpha})$ by~\eqref{eq:tau}.
	By Proposition~\ref{prop:CMT},
	$\lim_{n\to\infty}\sum_{j=0}^n e_{n-j}\, \mu(y\in Y:\tau(y)>j)=0$.
Hence
\[
 \int_Y g\circ f^n\,d\mu 
 = K\sum_{j=0}^n a_{n-j}^{-1}\int_Y 1_{\{\tau>j\}}g\circ f^j\,d\mu + o(1).
\]

Next, 
\[
 \int_Y 1_{\{\tau \ge (1+\eps^{-1})j\}}g\circ f^j\,d\mu
 \le |g|_\infty \mu_Y(\tau\ge \eps^{-1}j) \ll (\eps^{-1}j)^{-\alpha},
\]
so
\[
 \Big|\sum_{j=0}^n a_{n-j}^{-1}\int_Y 1_{\{\tau \ge (1+\eps^{-1})j\}} g\circ f^j\,d\mu \Big|
\ll \eps^\alpha \sum_{j=0}^n a_{n-j}^{-1} j^{-\alpha}\ll \eps^\alpha.
\]
(For the last inequality we applied Proposition~\ref{prop:CMT}.)

Also, by~\eqref{eq:tau},
\[
j^\alpha \mu_Y(\tau>j)-j^\alpha \mu_Y(\tau>(1+\eps)j)
\to\gamma(1-(1+\eps)^{-\alpha})
\]
so 
$\mu_Y(j<\tau\le (1+\eps)j)=O(\eps)j^{-\alpha}+o(j^{-\alpha})$.
Hence it follows from Proposition~\ref{prop:CMT} that
\begin{align*}
 \limsup_{n\to\infty}\Big| & \sum_{j=0}^n a_{n-j}^{-1} \int_Y 1_{\{\tau<j\le(1+\eps)j\}}g\circ f^j\,d\mu \Big|
\\ & \le |g|_\infty \limsup_{n\to\infty}\sum_{j=0}^n a_{n-j}^{-1} 
\mu_Y(j<\tau\le (1+\eps)j) \ll O(\eps).
\end{align*}

To summarise, we have shown that $\lim_{n\to\infty}\int_Y g\circ f^n\,d\mu=0$
if and only if
 $\lim_{n\to\infty}\sum_{j=0}^n a_{n-j}^{-1} \int_Y 1_{\{1+\eps)j< \tau < (1+\eps^{-1})j\}}g\circ f^j\,d\mu =0$.
This completes the proof of the first statement, and the second statement follows by Proposition~\ref{prop:CMT}.
\end{proof}

Fix $\eps>0$ and suppose that $j$ is sufficiently large that $\eps j\ge n_0$. Then
\begin{equation} \label{eq:p}
\int_{Y_{j+i}}g\circ f^j\,d\mu= \sum_{p=1}^q\sum_{\psi(r)=p}\int_{Y_{j+i,r}}g\circ f^j\,d\mu
\end{equation} 
where $f^j$ maps each $Y_{j+i,r}$ bijectively onto $X_{i,p}$.
Hence
our treatment divides naturally into two subcases depending on whether $g$ is mean zero or constant on $X_{n,p}$ for $n\ge n_0$, $1\le p\le q$.

\vspace{1ex}
\noindent
\textbf{Piecewise constant global observables.}
These are global observables $g\in L^\infty$ that are constant on $X_{n,p}$, i.e.\ $g|_{X_{n,p}}\equiv \bar g_{n,p}$, for all $n,p$, where the constants $\bar g_{n,p}$ are bounded and satisfy the constraint~\eqref{eq:g}.

\vspace{1ex}
\noindent
\textbf{Piecewise mean zero observables.}
These are $L^\infty$ observables $g$ such that 
$\int_{X_{n,p}}g\,d\mu=0$, i.e.\ $\bar g_{n,p}=0$, for all $n,p$.

\subsection{Piecewise constant global observables}
\label{sec:pwc}

In this subsection, we prove global-local mixing for piecewise constant global observables. (Condition~\eqref{eq:J} is not required for this.)
The proof divides into the cases $\alpha\in(0,1)$ and $\alpha=1$.

\begin{lemma} \label{lem:pwc}
Suppose $\alpha\in(0,1)$.  Assume
conditions~\eqref{eq:Y} and~\eqref{eq:K}.
Then global-local mixing holds for all piecewise constant centred global observables $g\in L^\infty$.
\end{lemma}

\begin{proof} 
Recall from Lemma~\ref{lem:K} that we need to prove
\[
\lim_{n\to\infty}\sum_{1\le j\le n} a_{n-j}^{-1} 
\sum_{\eps j < i < \eps^{-1}j}\int_{Y_{j+i}}g\circ f^j\,d\mu=0
\qquad\text{for all $\eps>0$}.
\]
For $\alpha\in(0,1)$, arguments similar to those in the proof of Lemma~\ref{lem:K} show that this reduces further to
\begin{equation} \label{eq:K2}
\lim_{n\to\infty}\sum_{\eps n< j< (1-\eps) n}
a_{n-j}^{-1} \sum_{\eps j < i < \eps^{-1}j}\int_{Y_{j+i}}g\circ f^j\,d\mu=0
\qquad\text{for all $\eps>0$}.
\end{equation} 
Indeed the contribution from $j=1$ to $\eps n$ is bounded by
\[
|g|_\infty\sum_{1\le j\le \eps n}
a_{n-j}^{-1} \mu_Y(\tau>j)
\ll n^{-(1-\alpha)}\sum_{1\le j\le \eps n}j^{-\alpha}
\ll \eps^{1-\alpha},
\]
and the contribution from $j=(1-\eps)n$ to $n$ is bounded by
\[
|g|_\infty\sum_{(1-\eps)n\le j\le n}
a_{n-j}^{-1} \mu_Y(\tau>j)
\ll n^{-\alpha} \sum_{(1-\eps)n\le j\le n}a_{n-j}^{-1}
\ll n^{-\alpha} \sum_{1\le j\le \eps n} j^{-(1 - \alpha)}
\ll \eps^{\alpha}.
\]

Let $g\in L^\infty$ be a piecewise constant centred global observable.
Using~\eqref{eq:p},
\[
\int_{Y_{j+i}}g\circ f^j\,d\mu=\sum_{p=1}^q\sum_{\psi(r)=p}\bar g_{i,p} \mu(Y_{j+i,r}).
\]
This combined with~\eqref{eq:K2} means that we have to show that $\lim_{n\to\infty} A_n=0$ where
\[
A_n=
\sum_{i\in R_n} \sum_{p=1}^q\bar g_{i,p}
\sum_{j\in S_{i,n}}
a_{n-j}^{-1}\sum_{\psi(r)=p}\mu(Y_{j+i,r}).
\]
Here, 
\[
R_n\subset\big\{\eps^2 n< i<(\eps^{-1}-1)n\big\}, \qquad
S_{i,n}\subset\big\{\eps \max\{i,n\}<j \le \min\{\eps^{-1}i, (1-\eps)n\}\big\}.
\]
(The containments might be strict since $i$ and $j$ are integers.)

\vspace{1ex} \noindent
\textbf{Step 1: Summation by parts over $j$.}
Let $\phi_{i,n}=\min S_{i,n}$, $\Phi_{i,n}=\max S_{i,n}$. 
Define
$\mu_{k,p}=\sum_{\psi(r)=p}\sum_{n\ge k}\mu(Y_{n,r})$.

Then
$A_n=
\sum_{i\in R_n}\sum_{p=1}^q\bar g_{i,p} B_{i,n,p}$ where
\[
B_{i,n,p}  = \sum_{j\in S_{i,n}} a_{n-j}^{-1}\sum_{\psi(r)=p}\mu(Y_{j+i,r})=
\sum_{\phi_{i,n}\le j\le \Phi_{i,n}} a_{n-j}^{-1}\{\mu_{j+i,p}-\mu_{j+i+1,p}\}.
\]
Summation by parts gives
\[
B_{i,n,p}  = 
\sum_{\phi_{i,n}\le j\le \Phi_{i,n}} \{a_{n-j}^{-1}-a_{n-j+1}^{-1}\}\mu_{j+i,p}
\;+\; \partial_{i,n,p,\phi_{i,n}} \;-\; \partial_{i,n,p,\Phi_{i,n}+1}
\]
where the boundary terms are 
\[
\partial_{i,n,p,c_{i,n}}=a_{n-c_{i,n}+1}^{-1}\,\mu_{c_{i,n}+i,p},
\qquad c_{i,n}=\phi_{i,n},\,\Phi_{i,n}+1.
\]

\vspace{1ex} \noindent
\textbf{Step 2: Simplification of $B_{i,n,p}$.}
Let $\tilde \mu_k=k^{-\alpha}$.
By~\eqref{eq:Y}, 
$\mu_{k,p}=\gamma_p\tilde\mu_k+o(k^{-\alpha})$.
Using also that $a_{n-c_{i,n}+1}^{-1}\ll n^{-(1-\alpha)}$,  
we obtain
\[
\partial_{i,n,p,c_{i,n}} = \gamma_p \tilde\partial_{i,n,c_{i,n}} +o(n^{-1}).
\qquad
\tilde\partial_{i,n,c_{i,n}} =a_{n-c_{i,n}+1}^{-1} \tilde\mu_{c_{i,n}+i}.
\]
Similarly, using in addition that
$a_{n-j}^{-1}-a_{n-j+1}^{-1} \ll
n^{-(2-\alpha)}$ and $\Phi_{i,n}-\phi_{i,n}\ll n$,
\[
\sum_{\phi_{i,n}\le j\le \Phi_{i,n}} \{a_{n-j}^{-1}-a_{n-j+1}^{-1}\}\mu_{j+i,p}
=\gamma_p \sum_{\phi_{i,n}\le j\le \Phi_{i,n}} \{a_{n-j}^{-1}-a_{n-j+1}^{-1}\}\tilde\mu_{j+i}
+o(n^{-1}).
\]
Hence, $B_{i,n,p}=\gamma_p \tilde B_{i,n}+o(n^{-1})$ where
\[
\tilde B_{i,n}  = 
\sum_{\phi_{i,n}\le j\le \Phi_{i,n}} \{a_{n-j}^{-1}-a_{n-j+1}^{-1}\}\tilde\mu_{j+i}
\;+\; \tilde \partial_{i,n,\phi_{i,n}} \;-\; \tilde\partial_{i,n,\Phi_{i,n}+1}
\]
Moreover, reversing the summation by parts argument in Step~1, we see immediately that
\[
\tilde B_{i,n}=
\sum_{\phi_{i,n}\le j\le \Phi_{i,n}} a_{n-j}^{-1}\{\tilde\mu_{j+i}-\tilde\mu_{j+i+1}\}=
\alpha \sum_{\phi_{i,n}\le j\le \Phi_{i,n}} a_{n-j}^{-1}(j+i)^{-(\alpha+1)}
+O(n^{-2}).
\]
Since 
$|R_n|\ll n$ and $|\bar g_{i,p}|\le |g|_\infty$, 
\begin{align*}
A_n & =\sum_{i\in R_n}\sum_{p=1}^q\bar g_{i,p}\Big\{\gamma_p\alpha \sum_{\phi_{i,n}\le j\le \Phi_{i,n}} a_{n-j}^{-1}(j+i)^{-(\alpha+1)} + o(n^{-1})\Big\}
\\ & =\alpha \sum_{i\in R_n}\bar g_i \sum_{\phi_{i,n}\le j\le \Phi_{i,n}} a_{n-j}^{-1}(j+i)^{-(\alpha+1)} \;+\; o(1).
\end{align*}

\vspace{1ex} \noindent
\textbf{Step 3: Summation by parts over $i$.}
We have reduced to showing that $\lim_{n\to\infty}C_n=0$ where
\[
C_n=\sum_{i\in R_n}\bar g_i \sum_{\phi_{i,n}\le j\le \Phi_{i,n}} a_{n-j}^{-1}(j+i)^{-(\alpha+1)} 
=\sum_{\eps n<j<(1-\eps)n} a_{n-j}^{-1}
\sum_{\eps j<i<\eps^{-1}j}\bar g_i (j+i)^{-(\alpha+1)} .
\]
Define
	\[
b_i=\sum_{k=n_0}^i k^{-\alpha}\bar g_k, \qquad  
b_n^*=\sup_{i\le n}|b_i|, \qquad
d_{i,j}= 
i^\alpha(j+i)^{-(\alpha+1)} .
\]
Then $d_{i,j}\ll n^{-1}$ and $d_{i,j}-d_{i+1,j}\ll n^{-2}$.
Also, 
$b_n^*=o(a_n)=o(n^{1-\alpha})$ by Lemma~\ref{lem:g}.

Summation by parts over $i$ gives
\begin{align*}
\sum_{\eps j<i<\eps^{-1}j} \bar g_i (j+i)^{-(\alpha+1)}
& =\sum_{\eps j<i<\eps^{-1}j} \bar g_i i^{-\alpha} d_{i,j}
=\sum_{\eps j<i<\eps^{-1}j} (b_i-b_{i-1}) d_{i,j}
\\ &  =\sum_{\eps j<i<\eps^{-1}j} b_i (d_{i,j}-d_{i+1,j})
\;+\; O(b_n^* n^{-1}) 
\\ &\ll n\cdot b_n^* \cdot n^{-2}+b_n^* n^{-1}=2 b_n^* n^{-1}=o(n^{-\alpha}).
\end{align*}
Hence $|C_n|\ll n\cdot n^{-(1-\alpha)}\cdot o(n^{-\alpha})=o(1)$,
completing the proof.
\end{proof}

\begin{lemma} \label{lem:pwc1}
Suppose $\alpha=1$.  Assume
conditions~\eqref{eq:Y} and~\eqref{eq:K}.
Then global-local mixing holds for all piecewise constant centred global observables $g\in L^\infty$.
\end{lemma}

\begin{proof}
Notice that
\[
\sum_{(1-\eps)n\le j\le n}a_{n-j}^{-1}\mu_Y(\tau>j)
\ll n^{-1} \sum_{(1-\eps)n\le j\le n}1 \le \eps,
\]
so the contribution from $(1-\eps)n\le j\le n$ is negligible as in
the proof of Lemma~\ref{lem:pwc}.
However, the contribution from $1\le j\le \eps n$ cannot be immediately ignored.
Also, it is convenient to include the contribution from
$1\le i\le \eps j$.
Hence, we initially reduce to showing that
\[
\lim_{n\to\infty}\sum_{1\le j< (1-\eps) n}
a_{n-j}^{-1} \sum_{1\le i < \eps^{-1}j}\int_{Y_{j+i}}g\circ f^j\,d\mu=0
\qquad\text{for all $\eps>0$}.
\]

\vspace{1ex} \noindent
\textbf{Step 1. Simplifying $a_{n-j}^{-1}$.}
The main difference from the proof of Lemma~\ref{lem:pwc} is that we can immediately replace $a_{n-j}^{-1}$ by $a_n^{-1}$. (This idea seems unhelpful when $\alpha<1$.)
Note that $a_{n-j}^{-1}-a_n^{-1}\ll (\log n)^{-2}j/n \le (\log n)^{-2}$.
Hence 
\begin{align*}
\Big|\sum_{1\le j< (1-\eps) n}
\big\{a_{n-j}^{-1}-a_n^{-1}\big\} \sum_{1\le i < \eps^{-1}j}\int_{Y_{j+i}}g\circ f^j\,d\mu\Big|
& \ll (\log n)^{-2}|g|_\infty \sum_{1\le j\le n} \mu_Y(\tau>j)
\\ & \ll (\log n)^{-1}.
\end{align*}
This means that we reduce to showing 
\[
\lim_{n\to\infty} (\log n)^{-1} 
\sum_{1\le j\le n}\, \sum_{1\le i <  \eps^{-1}j}\int_{Y_{j+i}}g\circ f^j\,d\mu=0,
\]
or equivalently
$\lim_{n\to\infty}(\log n)^{-1} A_n=0$ where
\[
A_n  =
\sum_{1\le i\le \eps^{-1}n}\, \sum_{\eps i \le j\le n}
\sum_{p=1}^q \sum_{\psi(r)=p} \int_{Y_{j+i,r}}g\circ f^j\,d\mu
 =\sum_{1\le i\le \eps^{-1}n}\sum_{p=1}^q \bar g_{i,p}
 \sum_{\psi(r)=p} \,\sum_{\eps i \le j\le n} \mu(Y_{j+i,r}).
\]
By~\eqref{eq:Y},
\begin{align*}
A_n & =
\sum_{1\le i\le \eps^{-1}n}\sum_{p=1}^q \bar g_{i,p}
 \Big(\sum_{\psi(r)=p} \,\sum_{j\ge \eps i} \mu(Y_{j+i,r})
 -\sum_{\psi(r)=p} \,\sum_{j>n} \mu(Y_{j+i,r})
\Big)
 \\ & = \sum_{1\le i\le \eps^{-1}n}\sum_{p=1}^q \bar g_{i,p}
  \big(\gamma_p i^{-1}+o(i^{-1})+O(n^{-1})\big)
  = \sum_{1\le i\le \eps^{-1}n}\sum_{p=1}^q \bar g_{i,p}
  \gamma_p i^{-1} \;+\;o(\log n)
 \\ &  = \sum_{1\le i\le \eps^{-1}n}\bar g_i i^{-1} \;+\;o(\log n)
  = o(\log n),
\end{align*}
where the last estimate follows from Lemma~\ref{lem:g}.
\end{proof}

\subsection{Piecewise mean zero observables}
\label{sec:pw0}

In this subsection, we prove global-local mixing for piecewise mean zero observables.

\begin{lemma} \label{lem:pw0}
Assume conditions~\eqref{eq:Y},~\eqref{eq:K} and~\eqref{eq:J}.
Then global-local mixing holds for all piecewise mean zero observables $g\in L^\infty$.
\end{lemma}

\begin{proof}
Let $c_{j,i,p}$ be as in~\eqref{eq:J}.
Since $g$ is piecewise mean zero,
\[
\sum_{\psi(r)=p}\int_{Y_{j+i,r}}g\circ f^j\,d\mu=
\int_{X_{i,p}}gJ_{j,i,p}\,d\mu
=\int_{X_{i,p}}g(J_{j,i,p}-c_{j,i,p})\,d\mu,
\]
so by Proposition~\ref{prop:X},
\[
\Big|\sum_{\psi(r)=p}\int_{Y_{j+i,r}}g\circ f^j\,d\mu\Big|
\ll j^{-\alpha} |g|_\infty \sup_{X_{i,p}}|J_{j,i,p}-c_{j,i,p}|
\]
for all $\eps j< i< \eps^{-1}j$.
By~\eqref{eq:J},
\[
\lim_{j\to\infty}\;  j^\alpha \!\!\!\!\!\!\sum_{\eps j<i< \eps^{-1}j} \sum_{\psi(r)=p}\int_{Y_{j+i,r}}g\circ f^j\,d\mu=0.
\]
By~\eqref{eq:p}, we obtain
$\lim_{j\to\infty} j^\alpha \sum_{\eps j<i< \eps^{-1}j}\int_{Y_{j+i}}g\circ f^j\,d\mu=0$
so the result follows from Lemma~\ref{lem:K}.
\end{proof}

\section{Examples}
\label{sec:ex}

In this section, we verify the hypotheses of Theorem~\ref{thm:main} for a variety of examples, thereby proving full global-local mixing.
Most of the examples 
 are AFN maps~\cite{Zweimuller98,Zweimuller00} and the necessary background is contained in
Subsection~\ref{sec:AFN}. 
In Subsection~\ref{sec:one}, we consider nonMarkovian maps with one neutral fixed point.
In Subsection~\ref{sec:several}, we consider intermittent maps with several neutral fixed points.
In Subsection~\ref{sec:EMV}, we treat a higher-dimensional nonMarkovian intermittent map introduced in~\cite{EMV21}.
In Subsection~\ref{sec:parabolic}, we consider parabolic rational maps of the complex plane.

\subsection{Background on AFN maps}
\label{sec:AFN}

Zweim\"uller~\cite{Zweimuller98,Zweimuller00}  studied a class of nonMarkovian interval maps
$f : X \to X$ ($X\subset\R$ a closed bounded interval) with indifferent fixed points. It is assumed that there
is a measurable partition $\cP$ of $X$ into open intervals such that $f$ is $C^2$ and
strictly monotone on each $I\in\cP$. 
The map is said to be AFU if:

\begin{itemize}

\parskip=-2pt
\item[(A)] \emph{Adler’s condition:} $f''/(f')^2$ is bounded on $\bigcup_{I\in\cP}I$;
\item[(F)] \emph{Finite images:} $\{f I : I \in\cP\}$ is finite;
\item[(U)] \emph{Uniform expansion:} There exists $\rho>1$ such that
$|f'| \ge\rho$ on $\bigcup_{I\in\cP}I$.
\end{itemize}
The map is AFN if condition (U) is relaxed to condition (N) where $|f'|\ge1$ and there is a finite number of neutral fixed points $\xi_1,\dots,\xi_d$ satisfying certain properties such that $|f'|>1$ on $X\setminus\{\xi_1,\dots,\xi_k\}$. In our examples,  there
exist $\alpha\in(0,1]$ and $b_1,\dots,b_d>0$ such that
\begin{equation} \label{eq:xi}
|f(x)-x|\sim b_k|x-\xi_k|^{1+1/\alpha}, \qquad
|f'(x)-1|\sim b_k(1+1/\alpha)|x-\xi_k|^{1/\alpha}, 
\end{equation}
as $x\to\xi_k$, $k=1,\dots, d$.
Condition (N) is shown in~\cite{Zweimuller98,Zweimuller00} to be satisfied for such maps.

We will make use of the
  following basic distortion estimate.
  Let $M=\sup | f'' |/(f')^2$. By Adler's condition, $M<\infty$.

\begin{prop} \label{prop:BD}
  $|f'(x)^{-1}- f'(y)^{-1}| \le  Me^{2M} | x - y|$
  for all $x,y\in I$ and all $I\in\cP$.
\end{prop}

\begin{proof}
  First, note that
  $| f'(x) / f' (y) | \le  e^M$ for all $x,y\in I$ (see \cite[Section 4.3]{Aaronson}
  for details).
  Hence for $x,y\in I$ with $x<y$,
  \[
    \frac{ f' (x) - f' (y) }{ f'(x)f'(y) }
    = \frac{ f'' ( z ) }{ f' (z )^2 }
    \frac{ f' ( z )  }{ f' (x) } \frac{ f'( z ) }{ f'(y) }
    ( y - x )
    \le  Me^{2M} ( y - x )
  \]
  for some $z \in (x,y)$ by the mean value theorem.
\end{proof}

Topologically mixing AFN maps are conservative and exact with respect to Lebesgue measure and admit a unique (up to scaling) equivalent $\sigma$-finite invariant measure $\mu$.
When there is at least one neutral fixed point, $\mu(X)=\infty$.
Moreover $\mu(X\setminus \bigcup_{k=1}^d B_\eps(\xi_1))<\infty$ for all $\eps>0$.

\paragraph{Global observables}
For AFN maps, it is natural to define $g\in L^\infty$ to be a global observable if
there exists $\bar g\in\R$ such that
\[
\lim_{\eps\to0^+}\frac{\int_{X\setminus\bigcup_{k=1}^d B_\eps(\xi_k)}g\,d\mu}{
\mu({X\setminus\bigcup_{k=1}^d B_\eps(\xi_k)})}=\bar g.
\]
For our examples, it is easily seen that this is equivalent to our abstract Definition~\ref{def:global}.

\paragraph{Verification of conditions~\eqref{eq:Y},~\eqref{eq:K} and~\eqref{eq:J} for AFN maps}

We induce to an appropriate set $Y$ with $0<\mu(Y)<\infty$ bounded away from the neutral fixed points and normalise so that $\mu(Y)=1$.
By general arguments~\cite{Zweimuller98,Zweimuller00}, the induced map
$F=f^\tau:Y\to Y$ is an AFU map; the properties (A) and (F) are inherited from 
$f$ and clearly $|F'|\ge |f'|$ which is uniformly larger than $1$ on $Y$.
With extra care, we are able to ensure that $F$ is topologically mixing and hence mixing.
Standard arguments, repeated in the examples below, show that~\eqref{eq:Y}  is satisfied; moreover 
$\mu(Y_n)\sim \alpha \gamma n^{-(\alpha+1)}$ for some $\gamma>0$.
Hence condition~\eqref{eq:K} is satisfied by~\cite{Gouezel11,MT12}.

It follows that the main effort focuses on condition~\eqref{eq:J}.
One ingredient is the following:

\begin{prop} \label{prop:h}
For AFN maps, the density 
$h=d\mu/d\Leb$ has the property that
the one-sided limits $h(y^\pm)\in(0,\infty)$
exist for all $y\in Y$.
\end{prop}

\begin{proof}
Since $F$ is an AFU map, 
$h|_Y$ lies in $\BV(Y)$ and is bounded away from zero~\cite{Rychlik83}.
This implies the result.
\end{proof}

\begin{rmk} \label{rmk:thaler}
In the special case where the branches $f|_I$ are full for each $I\in\cP$
(so $\overline{fI}=X$ for $I\in\cP$), these maps were studied by Thaler~\cite{Thaler80,Thaler83}.
It is then the case that $F:Y\to Y$ is an AFU map with full branches. (In particular, $F$ is Gibbs-Markov.) 
The density $h$ has the property that $h|_Y$ is Lipschitz (see for example~\cite[Proposition~2.5]{KKM19}).  More generally, if
$F$ is a Gibbs-Markov map, then $h|_Y$ is uniformly piecewise Lipschitz on the image partition (namely the partition of $Y$ generated by the images of the branches of $F$).
\end{rmk}

Next, we mention the following convenient formula for 
$J_{j,n,p}$ from Lemma~\ref{lem:pw0} expressed in terms of
$h|_Y$. (This result does not use the AFN structure.)

\begin{prop} \label{prop:J}
For $x\in X_{n,p}$,
\[
J_{j,n,p}(x)=\frac{
\sum_{\psi(r)=p}h(y_{j,n,p})\big|(f^j)'(y_{j,n,r})\big|^{-1}
}
{
\sum_{\psi(r)=p}\sum_{\ell=1}^\infty h(y_{\ell,n,r})\big|(f^\ell)'(y_{\ell,n,r})\big|^{-1}
}
\]
where $y_{\ell,n,r}\in Y_{n+\ell,r}$ with $f^\ell y_{\ell,n,r}=x$.
\end{prop}

\begin{proof} 
It follows from the usual formula 
\(
\mu=\sum_{\ell=0}^\infty f^\ell_*(\mu|_{\{y\in Y:\tau(y)>\ell\}})
\)
for constructing $\mu$ on $X$ from $\mu|_Y$ that
\[
J_{j,n,p}=\displaystyle d\sum_{\psi(r)=p}(f^j|_{Y_{j+n,r}})_*\mu \,\Big/\,
d\sum_{\ell=1}^\infty\sum_{\psi(r)=p}(f^\ell|_{Y_{\ell+n,r}})_*\mu
\]
for $j,n,p\ge1$.
The result is a direct consequence of this.
\end{proof}

\begin{prop} \label{prop:JJ} Let $B>0$. Suppose that uniformly in $r\in\bbA$,
\begin{itemize}
\item[(a)]
$\lim_{n\to\infty}\sup_{Y_{n,r}}|h-h_r|=0$;
\item[(b)]
\(
\sup_{\ell,\,n\ge1}\sup_{Y_{\ell+n,r}}
 \Big({\displaystyle \frac{ n }{ \ell+n }} \Big)^{ B + 1 }|( f^\ell )' |
          <\infty
\); and
\item[(c)]
\(
         \lim_{n\to\infty}\sup_{1\le \ell\le Kn}\sup_{Y_{\ell+n,r}}\Big|\Big( {\displaystyle \frac{ n }{ \ell+n }} \Big)^{ B + 1 }
 |( f^\ell )' |
          -\omega_r\Big|=0
\)
for all fixed $K\ge1$.
\end{itemize}
Then condition~\eqref{eq:J} holds with $c_{j,n,p}=B n^B (j+n)^{-(B+1)}$.
\end{prop}

\begin{proof}
We take $K>\eps^{-1}$.
By (a) and (c),
\[
h(y_{\ell,n,r})|(f^\ell)'(y_{\ell,n,r})|^{-1}= h_r \omega_r^{-1}n^{B+1}(\ell+n)^{-(B+1)}(1+e_{\ell,n,r}(x))
\]
where $\lim_{n\to\infty}e_{\ell,n,r}(x)=0$ uniformly in $1\le \ell\le Kn$, $r$ 
and $x\in X_{n,p}$.
Hence by (b),
\begin{align*}
\sum_{\ell\ge 1} h(y_{\ell,n,r})|(f^\ell)'(y_{\ell,n,r})|^{-1}
&= \sum_{\ell=1}^{Kn} h(y_{\ell,n,r})|(f^\ell)'(y_{\ell,n,r})|^{-1}
+\sum_{\ell>Kn} h(y_{\ell,n,r})|(f^\ell)'(y_{\ell,n,r})|^{-1}
\\ & =B^{-1} h_r\omega_r^{-1}n \big\{1+\tilde e_{n,r}(x) + O(K^{-B})\big\}
\end{align*}
where $\lim_{n\to\infty}\tilde e_{n,r}(x)=0$ uniformly in $r$ and
$x\in X_{n,p}$.
Since $K$ is arbitrarily large,
\[
J_{j,n,p}= \frac
{B n^B (j+n)^{-(B+1)}(1+e_{j,n,r})}
{1+\tilde e_{n,r}}
\]
where $\lim_{n\to\infty}e_{j,n,r}(x)=0$ uniformly in $j< \eps^{-1}n$, $r$ 
and $x\in X_{n,p}$.
We obtain
\[
J_{j,n,p}-c_{j,n,p}=B n^B (j+n)^{-(B+1)}\Big\{
\frac{e_{j,n,r}-\tilde e_{n,r}} {1+\tilde e_{n,r}}  \Big\}.
\]
Hence
\[
\sum_{\eps j<n\le \eps^{-1}j}|J_{j,n,p}-c_{j,n,p}|\ll \sup_{\eps j<n\le \eps^{-1}j}
\Big| \frac{e_{j,n,r}-\tilde e_{n,r}}
{1+\tilde e_{n,r}} \Big|
\]
and~\eqref{eq:J} follows.
\end{proof}

\subsection{Maps with one neutral fixed point}
\label{sec:one}

We begin with examples where there is one neutral fixed point at $0$
(so $f'(0)=1$) and $f'>1$ on $(0,1]$.
Special cases are the Farey map, generalised Pomeau-Manneville maps and generalised LSV maps considered in~\cite{BonannoGiuliettiLenci18,BonannoLenci21}.
We do not assume full branches, nor even that the maps are Markov.
For the moment, we restrict to finitely many branches; this restriction is relaxed in Example~\ref{ex:thaler-D>d}.

\begin{example} \label{ex:2branch}
Fix $\alpha\in(0,1]$, $\kappa>1/\alpha$, $b>0$, $\eta_1\in(0,1)$.
We consider piecewise smooth maps $f:[0,1]\to [0,1]$ given by
$f(x)=\begin{cases} f_0(x), & x<\eta_1 \\ f_1(x), & x>\eta_1\end{cases}$
with two orientation preserving branches satisfying the following properties:
The second branch $f_1$ is uniformly expanding and maps $[\eta_1,1]$ onto $[0,1]$.
The first branch $f_0$ is uniformly expanding on $[\eps,\eta_1]$ for all $\eps>0$ with a neutral fixed point at $0$:
\[
f_0(x)=x + b x^{1+1/\alpha}+O(x^{1+\kappa}), \qquad
f_0'(x)=1 + b(1+1/\alpha)x^{1/\alpha}+O(x^\kappa).
\]
We assume that $f_0\eta_1=\eta \in (f_1^{-1}\eta_1,1]$.

\begin{rmk}
This includes the LSV map as the special case $f_0(x)=x(1+2^{1/\alpha}x^{1/\alpha})$, $f_1(x)=2x-1$. However, in general $f$ is not full-branch.
Even when $f$ is full-branch, these maps are only covered in~\cite{BonannoLenci21} under an extra monotonicity condition~(A5).
\end{rmk}

This is a simple example of a topologically mixing AFN map with invariant set $X=[0,\eta]$.
Note that the first branch $f_0:[0,\eta_1]\to X$ is full and the second uniformly expanding branch $f_1:[\eta_1,\eta]\to X$ need not be full.
As discussed in Subsection~\ref{sec:AFN}, $f$ is conservative and exact, there is a unique (up to scaling) $\sigma$-finite invariant measure~$\mu$ equivalent to Lebesgue, and $\mu(X)=\infty$.

We take $Y=[\eta_1,\eta]$.  The induced AFU map $F=f^\tau:Y\to Y$ has at most one short branch (on 
$Y_1=[f_1^{-1}\eta_1,\eta]$) and the remaining branches are full.
It is easily seen that $F$ is topologically mixing and hence mixing.

There exists $n_0\ge1$ (depending only on the size of the short branch) such that
$f^j$ maps $Y_{j+n}$ bijectively onto
$X_n$ for all $j\ge1$, $n\ge n_0$.
In particular, we can take $q=|\bbA|=1$ in the abstract setup of Section~\ref{sec:main} and we suppress the subscripts $p$, $r$.

Write $X_n=[x_{n+1},x_n]$ where $x_n\downarrow0$.
Without loss of generality, we can suppose that $\kappa\in(1/\alpha,1+1/\alpha)$.
Let $\kappa'=\alpha\kappa-1\in(0,\alpha)$.

\begin{prop} \label{prop:xn}
Let $b'=\alpha^\alpha b^{-\alpha}$ and $b''=\alpha^{\alpha+1}b^{-\alpha}$.
Then
$x_n=  b' n^{-\alpha}+O(n^{-(\alpha+\kappa')})$
and $x_{n-1}-x_n= b'' n^{-(\alpha+1)}+O(n^{-(\alpha+\kappa'+1)})$.
\end{prop}

\begin{proof}
We recall the standard argument (see for example~\cite[Lemma~4.8.6]{Aaronson}).
Note that
\[
x_{n-1}^{-1/\alpha}
 =(f_0x_n)^{-1/\alpha}
=\{x_n(1+bx_n^{1/\alpha}+O(x_n^\kappa))\}^{-1/\alpha}
 =x_n^{-1/\alpha}(1-\alpha^{-1}bx_n^{1/\alpha}+O(x_n^\kappa)).
\]
Hence,
\begin{equation} \label{eq:xn}
x_n^{-1/\alpha}-x_{n-1}^{-1/\alpha}=
\alpha^{-1}b+O(x_n^{\kappa-1/\alpha}).
\end{equation}
Summing~\eqref{eq:xn}, we obtain the rough estimate $x_n^{-1/\alpha}\approx n$ and substituting this into~\eqref{eq:xn} yields the more precise expression
\[
x_n^{-1/\alpha}-x_{n-1}^{-1/\alpha}=
\alpha^{-1}b+O(n^{-\kappa'}).
\]
Summing this gives
$x_n^{-1/\alpha}=\alpha^{-1}b n(1+O(n^{-\kappa'}))$ yielding the estimate for $x_n$.

Finally,
\begin{align*}
x_{n-1} & -x_n
 =f_0x_n-x_n=bx_n^{1+1/\alpha}+O(x_n^{1+\kappa})
\\ & =b(b')^{1+1/\alpha}n^{-(\alpha+1)}(1+O(n^{-\kappa'})+O(n^{-\alpha(1+\kappa))}) =b''n^{-(\alpha+1)}+O(n^{-(\alpha+\kappa'+1)}),
\end{align*}
completing the proof.
\end{proof}

By Proposition~\ref{prop:xn},
$\Leb(Y_n)\sim f_1'(\eta_1)^{-1}\Leb(X_n)\sim 
f_1'(\eta_1)^{-1} b''\,n^{-(\alpha+1)}$ as $n\to\infty$.
Since $h\in \BV(Y)$,
\[
\mu(Y_n)\sim h(\eta_1^+)\Leb(Y_n)
\sim h(\eta_1^+) f_1'(\eta_1)^{-1} b'' n^{-(\alpha+1)} 
\quad\text{as $n\to\infty$}
\]
 verifying condition~\eqref{eq:Y}.

It remains to verify condition~\eqref{eq:J}.

\begin{lemma} \label{lem:f'}
        The diffeomorphism $f^j|_{Y_{j+n}}:Y_{j+n}\to X_n$ satisfies
        \[
\{(f^j)'\}^{-1} =  \{f_1'(\eta_1)\}^{-1} \Big(\frac{n}{j+n}\Big)^{\alpha+1}
                \big( 1 + O(n^{-\kappa'})\big)
                ,
        \]
        where the implied constant is independent of $j$ and $n$ uniformly on $Y$.
\end{lemma}

\begin{proof}
Let $y\in Y_{j+n}$.
By Proposition~\ref{prop:xn}, $f^ky\in X_{j+n-k}$ satisfies
\[
f^ky= b' (j+n-k)^{-\alpha}+O((j+n-k)^{-(\alpha+\kappa')}).
\]
Note that
\[
\log f_0'(x)= \log(1+b(1+1/\alpha)x^{1/\alpha}+O(x^\kappa))=
b(1+1/\alpha)x^{1/\alpha}+O(x^\kappa).
\]
Since $b(b')^{1/\alpha}=\alpha$, we obtain
\begin{align*}
\log f_0'(f^ky) & =
(\alpha+1)(j+n-k)^{-1}\big\{1+O((j+n-k)^{-\kappa'})\big\} + O((j+n-k)^{-\alpha\kappa}) 
\\ & =
(\alpha+1)(j+n-k)^{-1}+O((j+n-k)^{-(1+\kappa')}),
\end{align*}
for $1\le k\le j$.
Hence
\begin{align*}
\log (f^j)'(y) & =\log f_1'(y)+ \sum_{k=1}^{j-1} \log f_0'(f^ky)
\\ & = \log f_1'(y)+(\alpha+1)\sum_{k=n+1}^{j+n-1} (k^{-1}+O(k^{-(1+\kappa')}))
\\ & = \log f_1'(y) +(\alpha+1)\log \frac{j+n}{n}  +O(n^{-\kappa'}).
\end{align*}
It follows that
\(
(f^j)'(y) =  f_1'(y) \big(\frac{j+n}n\big)^{\alpha+1}
	\big( 1 + O(n^{-\kappa'})\big)
	\)
and hence
\[
\{(f^j)'(y)\}^{-1} =  \{f_1'(y)\}^{-1} \Big(\frac{n}{j+n}\Big)^{\alpha+1}
	\big( 1 + O(n^{-\kappa'})\big).
	\]
Finally, by Proposition~\ref{prop:BD} and the estimate $\Leb( Y_n )\ll n^{-(\alpha+1)}$, 
\[
        | \{f'(y)\}^{ -1 }
        - \{ f'( \eta_1 )\}^{ -1 } |
        \ll | y - \eta_1|
        \ll \sum_{k = n + j}^\infty k^{ - (\alpha + 1 )}
        \ll n^{-\alpha}\le n^{-\kappa'}
\]
and the result follows.
\end{proof}

\begin{cor}\label{cor:J}
	Let $J_{j,n}$ be defined as in Lemma~\ref{lem:pw0}. Then,
\[
\lim_{j \to \infty }\sum_{\eps j<n\le \eps^{-1}j} \sup_{X_{n}}|J_{j,n}- \alpha n^\alpha(j+n)^{-(\alpha+1)}|=0.
\]
\end{cor}

\begin{proof}
Let $x\in X_n$, $y_{j,n}\in Y_{j+n}$ with $f^jy_{j,n}=x$. 
Note that $y_{j,n}\downarrow \eta_1$ as $n\to\infty$ uniformly in $j$ and $x\in X_n$.
Hence
$h(y_{j,n})\to h(\eta_1^+)$ as $n\to\infty$ uniformly in $j$ and $x\in X_n$ by
one-sided continuity of $h$.

Recall from Proposition~\ref{prop:J} that
\[
J_{j,n}(x)=\frac{
h(y_{j,n})\big\{(f^j)'(y_{j,n})\big\}^{-1}
}
{
\sum_{\ell=1}^\infty h(y_{\ell,n})\big\{(f^\ell)'(y_{\ell,n})\big\}^{-1}
}\;.
\]
By Lemma~\ref{lem:f'},
\[
h(y_{j,n})\{(f^j)'(y_{j,n})\}^{-1}\sim h(\eta_1^+)\{f_1'(\eta_1)\}^{-1}n^{\alpha+1}(j+n)^{-(\alpha+1)}
\]
as $n\to\infty$ uniformly in $j$ and $x\in X_n$, and
\[
\sum_{\ell=1}^\infty h(y_{\ell,n})\{(f^\ell)'(y_{\ell,n})\}^{-1}
\sim h(\eta_1^+)\{f_1'(\eta_1)\}^{-1}\alpha^{-1}n
\]
as $n\to\infty$ uniformly in $x\in X_n$. The result follows.
\end{proof}

By Corollary~\ref{cor:J}, condition~\eqref{eq:J} holds with
$c_{j,n}=\alpha n^\alpha(j+n)^{-(\alpha+1)}$, completing the proof of global-local mixing.
\end{example}

\begin{example} \label{ex:Qbranch}
We generalise the previous example to one with several uniformly expanding branches.

Fix $\alpha\in(0,1]$, $b>0$ and $0=\eta_0<\eta_1<\eta_2<\dots<\eta_{Q+1}=1$
where ${Q\ge1}$.
We suppose that the branches $f_r=f|_{[\eta_r,\eta_{r+1}]}$ 
are orientation preserving with $f_r(\eta_r)=0$ for all $r=0,\dots,Q$.
Suppose that $f_0$ has a neutral fixed point at $0$ as in Example~\ref{ex:2branch} and is uniformly expanding away from zero, and that the remaining branches $f_1,\dots,f_Q$ are uniformly expanding.
We require that $f_0\eta_1=\eta\in(f_Q^{-1}\eta_Q,1]$ and set $X=[0,\eta]$.

Again, 
$f:X\to X$ is conservative and exact with at least one full branch $f_0$.
There is a 
unique (up to scaling) $\sigma$-finite absolutely continuous invariant measure $\mu$, and $\mu(X)=\infty$.

We take $Y=[\eta_1,\eta]$ and note that $F=f^\tau:Y\to Y$ is a mixing AFU map.
In this example, 
$f^j|_{Y_{j+n}}$ is $Q$-to-$1$.
Setting $Y_{n,r}=Y_n\cap [\eta_r,\eta_{r+1}]$ for $1\le r\le Q$, we obtain for $n_0\ge1$ sufficiently large that $f^j$ maps $Y_{j+n,r}$ bijectively onto $X_n$ for all 
$j\ge1$, $n\ge n_0$, $1\le r\le Q$.
We take $q=1$, $\bbA=\{1,\dots,Q\}$ and suppress subscripts $p$.

Write $X_n=[x_{n+1},x_n]$ where $x_n\downarrow0$.
The statement and proof of Proposition~\ref{prop:xn} goes through as before.
In particular, $\Leb(Y_{n,r})\sim f_r'(\eta_r)^{-1}\Leb(X_n)\sim 
f_r'(\eta_r)^{-1} b''\,n^{-(\alpha+1)}$ 
(with $b''=\alpha^{\alpha+1}b^{-\alpha}$)
and 
$\mu(Y_{n,r})
\sim 
 h(\eta_r^+) f_1'(\eta_r)^{-1} b''
n^{-(\alpha+1)}$,
verifying condition~\eqref{eq:Y}.

Following the proof of Lemma~\ref{lem:f'}, we obtain that
the diffeomorphism $f^j|_{Y_{j+n,r}}:Y_{j+n,r}\to X_n$ satisfies
\[
	\{(f^j)'\}^{-1}= \{f_r'(\eta_r)\}^{-1} \Big(\frac{n}{j+n}\Big)^{\alpha+1}\big(1+O(n^{-\kappa'})\big).
\]
Substituting into Proposition~\ref{prop:J},
\[
\sum_{r=1}^Q h(y_{j,n,r})\{(f^j)'(y_{j,n,r})\}^{-1}\sim \sum_{r=1}^Q h(\eta_r^+)\{f_r'(\eta_r)\}^{-1}n^{\alpha+1}(j+n)^{-(\alpha+1)}
\]
as $n\to\infty$ uniformly in $j$ and $x\in X_n$, and
\[
\sum_{r=1}^Q\sum_{\ell=1}^\infty h(y_{\ell,n,r})\{(f^\ell)'(y_{\ell,n,r})\}^{-1}
\sim \sum_{r=1}^Q h(\eta_r^+)\{f_r'(\eta_r)\}^{-1}\alpha^{-1}n
\]
as $n\to\infty$ uniformly in $x\in X_n$.
Hence we again obtain that
$J_{j,n}\sim \alpha n^\alpha(j+n)^{-(\alpha+1)}$ as $n\to\infty$ uniformly in $j$ on $X_n$, verifying condition~\eqref{eq:J}.
\end{example}

\begin{example} \label{ex:genPM}
A set of examples parallel to those in Example~\ref{ex:Qbranch}
is given by AFN maps $f:[0,1]\to[0,1]$ with $Q+1$ branches, $Q\ge1$, and a neutral fixed point at $0$ as before, but now we assume that the first $Q$ branches are full and the last branch is possibly not full.
The analysis is the same as in Example~\ref{ex:Qbranch} with the simplification that $X=[0,1]$, so we omit the details.

This includes the case of Pomeau-Manneville maps 
$fx=x+b x^{1+1/\alpha} \bmod1$ for all $b\in[1,\infty)$. Note that the number of branches is given by~$\lceil b\rceil$.
Previously global-local mixing was studied by~\cite{BonannoLenci21} for the full-branch case with $b$ an integer, but we do not make this restriction.
\end{example}

\begin{example}\label{ex:Farey}
In the previous examples, we assumed for convenience that all branches were orientation preserving. We can easily handle a mixture of branches that preserve and reverse orientation. If a uniform branch $f_r=f|_{[\eta_r,\eta_{r+1}]}$ in Example~\ref{ex:Qbranch} is orientation reversing with $f_r(\eta_{r+1})=0$, then the arguments go through with
$f_r'(\eta_r)$ and $h(\eta_r^+)$ replaced by
$f_r'(\eta_{r+1})$ and $h(\eta_{r+1}^-)$.

In particular, we obtain full global mixing for the 
  Farey map $f : X \to X$. Here, $X=[0,1]$ and there are two full branches of opposite orientations given by
\[
      f_0(x)=\dfrac{ x }{ 1 - x },\quad  x \in [0,\tfrac12], \qquad\text{and}\qquad
      f_1(x)=\dfrac{ 1 - x }{ x },\quad  x \in (\tfrac12,1].
\]
This is an AFN map with a single neutral fixed point at $x=0$ and with $\alpha=1$.
\end{example}

\subsection{Maps with several neutral fixed points}
\label{sec:several}

In this subsection we consider full-branch AFN maps with $d\ge2$ equally sticky neutral fixed points $\xi_1,\dots,\xi_d$ satisfying~\eqref{eq:xi}.
In Example~\ref{ex:thaler}, we consider an example with~$d$ branches, each containing a neutral fixed point. 
In Example~\ref{ex:thaler-D>d}, we consider a generalisation where there can be finitely or infinitely many uniformly expanding branches in addition to the $d$ branches with neutral fixed points.

\begin{example}\label{ex:thaler}
Let $X=[0,1]$. Let $d\ge 2$ and  suppose that
$X=\bigcup_{k=1}^d I_k$ is a partition mod $0$ into $d$ open subintervals.
We suppose that the branches $f_k=f|_{I_k}$ are orientation preserving
and full for $k=1,\dots d$ with $f'\ge1$.
Suppose that the $d$ fixed points $\xi_1,\dots,\xi_d$ satisfy~\eqref{eq:xi} with $\alpha\in(0,1]$.
As usual, we assume Adler's condition and that $f'>1$ away from the fixed points.

We construct the inducing domain $Y$ following~\cite[Section~7]{Sera20} (see also~\cite[Section~4.1]{CMTsub}).
  The construction depends on whether $d = 2$ or $d \ge 3$. 
When $d = 2$, we take $Y = [ x, f x ]$ where $\{x, fx \}$ is the unique orbit of period $2$. When $d \ge 3$, we take $Y = \bigcup_{ k = 1 }^d (I_k \setminus  f^{ -1 } I_k )$. 
From now on, we focus on the case $d\ge3$.

The induced map $F=f^\tau:Y\to Y$ is topologically mixing and hence mixing (see for example~\cite[Proposition 4.2]{CMTsub}).
Moreover $F^3$ has full branches so it follows from Remark~\ref{rmk:thaler} that the density $h$ is Lipschitz.

We take $|\bbA|=d(d-1)$, $q=d$ and $n_0=1$.
Set 
\begin{equation} \label{eq:exX}
X_{n,p}=\{x\in X_n\cap I_p:f^nx\in I_p\},\quad 1\le p\le q,\;n\ge1.
\end{equation}
For each $1\le p\le q$, we note that
$Y_n\cap f^{-1}I_p$ is the disjoint union of $d-1$ intervals $Y_{n,p,k}\subset I_k$ where $1\le k\le d$, $k\neq p$.
To conform with the abstract setup, write
 $\bbA=\{r=(p,k):1\le k,p\le q,\,k\neq p\}$. 
Then we obtain decompositions 
$X_n=\bigcup_{p=1}^q X_{n,p}$, 
$Y_n=\bigcup_{p=1}^q \bigcup_{\psi(r)=p} Y_{n,r}$ where
$\psi^{-1}(p)=\{(p,k):1\le k\le q,\,k\neq p\}$ and $Y_{n,(k,p)}\subset I_k$.
  
  We can use Proposition~\ref{prop:xn} as before to show that
  $\Leb(X_{ n,p }) \sim b''_p n^{ - ( \alpha + 1 ) }$ where $b''_p = \alpha^{ \alpha + 1 }b_p^{ -\alpha }$.
  Let $\zeta_r$ be the accumulation point of the intervals $Y_{ n,r }$ and notice that $f \zeta_r =  \xi_p$. Thus, $\Leb(Y_{ n, r }) \sim f'(\zeta_r)^{ -1 } b''_p n^{ -(\alpha + 1) }$ and so
  $\mu (Y_{ n,r }) \sim h ( \zeta_r )  f'( \zeta_r )^{ -1 } b''_p n^{ - (\alpha + 1 ) } $ verifying condition~\eqref{eq:Y}.

  Following the proof of Lemma~\ref{lem:f'}, we find for each $p$, $k$ that the diffeomorphism $f^j : Y_{ j + n , r } \to X_{n,p}$ satisfies
	\[
		\{(f^j)'\}^{-1} =  \{f_r'(\zeta_r)\}^{-1} \Big(\frac{n}{j+n}\Big)^{\alpha+1}
		\big( 1 + O(n^{-\kappa'})\big)
		.
	\]
  Hence, proceeding in the same way as Corollary~\ref{cor:J}, we obtain for each $p$ that
  $J_{ j,n,p } \sim \alpha n^{ \alpha } ( j + n )^{ - ( \alpha + 1 )}$
as  $n\to\infty$ uniformly in $j$ on $X_{n,p}$,
  verifying~\eqref{eq:J}.
\end{example}

\begin{example}\label{ex:sing}
	We may also treat the interval maps studied in~\cite{CoaLuz2024,CoaLuzMub2023}. These are interval maps consisting of two full orientation preserving branches, each with a neutral fixed point. Unlike in Example~\ref{ex:thaler}, these maps are not AFN as each branch may have either a critical or singular point at the discontinuity. The set of maps we can treat is slightly more general than those introduced in~\cite{CoaLuzMub2023}. A precise definition of this class of maps is given in~\cite[Section 4.2]{CMTsub} with the slight modification that we impose the stronger condition~\eqref{eq:xi} for the expansion near the fixed points. The definition of the set $Y$, and the verification of~\eqref{eq:Y} and~\eqref{eq:K} can be found in~\cite[Section 4.2]{CMTsub}. With~\eqref{eq:Y} established, one can verify~\eqref{eq:J} in the same way as in Example~\ref{ex:thaler} with $d=q=|\bbA|=2$. 
\end{example}

\begin{example}\label{ex:thaler-D>d}
  Fix $D>d\ge1$, where we allow the possibility that $D=\infty$.
  Set $X=[0,1]$ with a partition mod $0$ into $D$ open subintervals $I_1, I_2,\ldots$.
  Assume that the branches $f_k=f|_{I_k}$ are orientation preserving and 
  full for each
  $k\ge1$ and that $f'\ge1$.
  For $k = 1, \ldots, d$ we suppose that the fixed points $\xi_k\in
  I_k$ satisfy~\eqref{eq:xi} for some $\alpha \in (0,1]$.
  As before we assume Adler's condition and that $f' x > 1$ for all
  $x \not \in \{ \xi_1, \ldots, \xi_d \}$.

  Set $Y = \big( \bigcup_{ i = 1 }^d (I_k \setminus  f^{ -1 } I_k)
  \big) \cup \big( \bigcup_{  i = d + 1 }^D I_k \big)$.
  (Unlike in Example~\ref{ex:thaler}, it is no longer necessary to
  distinguish low values of $d$.)

  We take $q = d$, $|\bbA|= d(D - 1)$, $n_0=1$.
  Define $X_{n,p}$ as in~\eqref{eq:exX}.
  For each $1\le p\le q$, note that
  $Y_n\cap f^{-1}I_p$ is the disjoint union of $D-1$ intervals
  $Y_{n,p,k}\subset I_k$ where $1\le k\le D$, $k\neq p$.
  In this way, we obtain decompositions $Y_n=\bigcup Y_{n,p,k}$ and
  $X_n=\bigcup X_{n,p}$ such that $f^j$ maps $Y_{j+n,p,k}$
  bijectively onto $X_{n,p}$ for all $j,n\ge1$, $1\le p\le q$,
  $1\le k\le D$, $k\neq p$.

  For $n\ge1$, $1\le p\le q$, $k\neq p$, we have  $F  Y_{ n,p,k }
  = f^n  Y_{ n,p,k }  = I_p  \setminus  f^{ -1 } I_p$,
  $F^2 Y_{n,p,k}=Y\setminus I_p$ and
  $F^3 Y_{n,p,k}=Y$.
  For the remaining partition elements $Y_{1,p,k}$ with $p>q$,
  $k\ge1$, we have
  $F Y_{1,p,k}=I_p$ and $F^2 Y_{1,p,k}=Y$.
  Hence $F$ is mixing.
  By Remark~\ref{rmk:thaler},
  the density $h$ is Lipschitz on the images of $F$.

  To calculate $\mu(Y_{ n,p,k })$ we proceed as in
  Example~\ref{ex:thaler} but this time we have to take care about
  the error terms which appear in the asymptotics 
  to deal with the case $D = \infty$.

  Write $|E|=\Leb(E)$ for $E\subset X$ measurable.

  \begin{cor} \label{cor:BD}
                $\sup_{I_k}(f')^{-1}\ll |I_k|$.
  \end{cor}

  \begin{proof}
    By the mean value theorem, there exists $y_k\in I_k$ such that
    $1=|X| = |fI_k| = f'(y_k)|I_k|$.
    The result follows from Proposition~\ref{prop:BD}.
  \end{proof}

        We emphasise that throughout this example, all the constants implied by the use of big O or $\ll$ are
        independent of $n$ and $k$ and the appropriate domain contained in~$X$.

  Fix $p\in\{1,\dots,d\}$ and let $r=(p,k)$ where $1\le k\le D$ and $k\neq p$.
  We can again use Proposition~\ref{prop:xn} to show that $| X_{ n,
  p  } | \sim b''_p n^{ -(\alpha + 1) }$ as $n\to\infty$.
  As in Example~\ref{ex:thaler}, let $\zeta_r \in I_k$ be the
  accumulation point of the intervals $Y_{ n,r }$.
  By the mean value theorem, there exists $y_{ n,r } \in Y_{n,r}$ such that
  $|X_{ n-1,p }| = |fY_{n,r}|=f'( y_{ n,r })|Y_{n,r}|$.
  By Corollary~\ref{cor:BD},
\[
    |Y_{n,r}|\ll |I_k| |X_{n-1,p}|\ll |I_k|n^{-(\alpha+1)}.
\]
  Hence, by Proposition~\ref{prop:BD},
  \begin{equation}\label{eq:sum-Yn}
    |  \{f' ( y_{ n,r })\}^{-1} - \{f'( \zeta_r)\}^{-1} |
    \ll | y_{n,r}- \zeta_r |
    \ll \sum_{ i = n }^{ \infty } |Y_{ i, r }|
    \ll |I_k|n^{-\alpha},
  \end{equation}
  and so
        \begin{equation}\label{eq:D>d-Y}
        |Y_{n,r}| = \{f'(y_{n,r})\}^{-1}|X_{n-1,p}|
=\{ f'( \zeta_r )\}^{-1}|X_{ n-1 , p }|+ O( |I_k|n^{-(\alpha+1)}).
        \end{equation}
  Since $h$ is Lipschitz on $Y$,
  \[
    \big| \mu ( Y_{ n, r } ) - h ( \zeta_r) | Y_{ n,r }| \big|
    \leq \int_{ Y_{ n,r } }| h - h ( \zeta_r )| \, d\Leb
    \ll | Y_{ n, r } | \sum_{ i = n }^{ \infty } |Y_{ i, r }|
    \ll |I_k|n^{-(2\alpha+1)}.
  \]
  Hence
\[
    \big| \mu ( Y_{ n, r } ) - h  ( \zeta_r)\{f'(\zeta_r)\}^{-1} | X_{ n-1,p }| \big|
    \ll |I_k|n^{-(2\alpha+1)}.
\]
  It follows that
  $\mu ( Y_{ n,r } ) \sim \gamma_r n^{ -( \alpha + 1) }$
  as $n\to\infty$ uniformly in $r$
  where $\gamma_r=h ( \zeta_r ) \{f'(\zeta_r)\}^{-1}b''_p$.
  Also $h ( \zeta_r ) \{f'(\zeta_r)\}^{-1}\ll |I_k|$ which is summable,
  so we have shown that for each $p\in\{1,\dots,q\}$,
  \[
    \sum_{k\neq p}\mu(Y_{n,p,k})
    =\omega_p|X_{n-1,p}|+O(n^{-(2\alpha+1)})
    \sim\omega_pb''_p n^{-(\alpha+1)}
    \quad\text{as $n\to\infty$},
  \]
  where $\omega_p=\sum_{k\neq p} h( \zeta_r )
  \{f'(\zeta_r)\}^{-1}\in(0,\infty)$.
  Hence
  we have verified condition~\eqref{eq:Y}.

  We now move on to verifying~\eqref{eq:J}.
        Again, we fix $p\in\{1,\dots,q\}$ and consider the $D-1$ indices
  $r=(p,k)$ where $k\neq p$.
  Let $x\in X_{ n,p }$  and let $y_{j,n,r} = y_{j,n,r}(x) \in Y_{ j + n , r }\in I_k$
  be such that $f^j y_{j,n,r} = x$.
        We first follow the proof of Lemma~\ref{lem:f'}.
Note that $f^ky_{j,n,r}=f_p^{-(j-\ell)}x$ for $k=1,\dots,j$ and hence is independent of $r$.
Hence, when applying Proposition~\ref{prop:xn}, we are justified in writing
        $\log f' (f^k  y_{j,n,r}) = (\alpha + 1) ( j + n - k)^{-1} ( 1 +  O ( (j + n - k)^{ - (1 + \kappa')})$ for $k=1,\dots,j$. This in turn yields 
  \[
    \log (f^j)' (y_{j,n,r})
    = \log f' ( y_{j,n,r})
    + ( \alpha + 1 ) \log \frac{ j + n }{ n } + O ( n^{-\kappa'} )
    .
  \]
  By Proposition~\ref{prop:BD} and~\eqref{eq:D>d-Y},
  \[
	  | \{f'(y_{j,n,r})\}^{ -1 }
    - \{ f' (\zeta_r )\}^{ -1 } |
    \ll | y_{j,n,r} - \zeta_r |
    \ll \{f' (\zeta_r)\}^{-1}\sum_{i = n + j}^\infty i^{ - (\alpha + 1 )}
    \ll \{f' (\zeta_r)\}^{-1} n^{-\alpha}
    .
  \]
  Thus, we obtain
\[
	  \{( f^j )' (y_{j,n,r})\}^{-1}
	  =\{f' \zeta_r )\}^{ -1 } 
	   \Big( \frac{ n }{ j +  n } \Big)^{ \alpha + 1 }
	  \big( 1 + O(n^{ - \kappa'}) \big)
    .
\]
Since $h$ is Lipschitz, it follows from~\eqref{eq:sum-Yn} that
  $|h(\zeta_r) - h ( y_{j,n,r}) | = O ( n^{-\alpha})$, so
  \[
	  h(y_{j,n,r}) \{(f^j)' (y_{j,n,r})\}^{ -1 }
    = n^{ \alpha + 1 } ( j + n )^{ -( \alpha + 1 ) }
    h( \zeta_r ) \{f'( \zeta_r )\}^{ -1 } ( 1 + O(n^{-\kappa'}) )
    .
  \]
  Hence
\[
	  \sum_{k \neq p } h(y_{j,n,r}) \{(f^j)' (y_{j,n,r})\}^{ -1 }
    = \omega_p n^{ \alpha + 1 } ( j + n )^{ -( \alpha + 1 ) } (
      1 +  O(n^{-\kappa'})
    ),
\]
and it follows that
\[
    \sum_{ \ell = 1 }^{ \infty }\sum_{k\neq p} 
	  h(y_{ \ell,n,r }) \{(f^{ \ell })' (y_{ \ell,n,r })\}^{ -1 }
=
    \omega_p \alpha^{ -1 } n ( 1 +  O(n^{-\kappa'})).
\]
       Hence 
$J_{ j,n,p } = \alpha n^{ \alpha } ( j + n )^{ - ( \alpha + 1 ) }( 1 +  O(n^{-\kappa'}))$, verifying condition~\eqref{eq:J}.
\end{example}

\subsection{A multidimensional intermittent map}
\label{sec:EMV}

\begin{example} \label{ex:EMV}
We consider a family of multidimensional intermittent maps
introduced in~\cite{EMV21}.
They are nonMarkovian and highly nonconformal (exponentially expanding in one direction and intermittent in the other direction).
We stick to the main case in~\cite{EMV21} of two-dimensional maps
$f:[0,1]\times\T\to[0,1]\times\T$ (where $\T=\R/\Z$), but this generalises to higher-dimensional examples with a general uniformly expanding second branch as discussed in~\cite[Remark~1.7]{EMV21}.

The maps $f$ in~\cite{EMV21} take the form 
$f(x,\theta)=f_0(x,\theta)$ for $0\le x\le \frac34$ and 
$f(x,\theta)=f_1(x,\theta)$ for $\frac34<x\le1$ where
\begin{align*}
f_0(x,\theta) & = (x(1+x^{1/\alpha}u(x,\theta)),\,4\theta\bmod1), \\
f_1(x,\theta) & = (4x-3,\,4\theta\bmod1) .
\end{align*}
Here, $u:[0,\frac34]\times\T\to(0,\infty)$ is a $C^2$ function satisfying
$u(0,\theta)\equiv c_0>0$ and $f_{0,1}(\frac34,\theta)>\frac{15}{16}$ for all $\theta\in\T$. In addition, it is assumed that $|x\partial_x u|_\infty$ and
$|\partial_\theta u|_\infty$ are sufficiently small. As usual, $\alpha\in(0,1]$.
This example has one neutral invariant circle $\Gamma=\{0\}\times \T$ with uniformly expanding dynamics on this circle.

As in the previous examples, we assume strengthened expansions near the neutral circle, namely that $u(x,\theta)=c_0+O(x^{\kappa-1/\alpha})$ uniformly in $\theta$, where $\kappa>1/\alpha$.

Let $X=\bigcup_{i=0}^3 f([0,\tfrac34]\times[\tfrac{i}{4},\tfrac{i+1}{4}])$.
By~\cite[Proposition~2.2]{EMV21}, $f$ restricts to a topologically exact map $f:X\to X$.
Moreover, by~\cite[Lemma~3.4]{EMV21}, there is a unique (up to scaling) absolutely continuous $f$-invariant $\sigma$-finite measure $\mu$ on $X$, and $\mu(X)=\infty$. 

For such maps, we prove global-local mixing, taking
$g\in L^\infty(X)$ to be a global observable if
there exists $\bar g\in\R$ such that
\[
\lim_{\eps\to0^+}\frac{\int_{X\setminus[0,\eps]\times\T}g\,d\mu}{
\mu(X\setminus[0,\eps]\times\T)}=\bar g.
\]

We induce on the set $Y=([\tfrac34,1]\times\T)\cap X$ noting that $\mu(Y)\in(0,\infty)$.
By~\cite[Lemma~3.2]{EMV21}, the induced map $F=f^\tau:Y\to Y$ is  mixing.
By~\cite[Proposition~4.11]{EMV21}, there exists $\gamma>0$ such that
$\mu(Y_n)\sim \alpha\gamma n^{-(\alpha+1)}$ verifying
condition~\eqref{eq:Y}.
Condition~\eqref{eq:K} follows directly from~\cite[Theorem~1.5(a)]{EMV21}.

\begin{prop} \label{prop:exact}
$f$ is conservative and exact.
\end{prop}

\begin{proof}
The first hit time $\tau:X\to\Z^+$ for $Y$ is shown directly in~\cite{EMV21} to be finite a.e., and
conservativity follows immediately by~\cite{Maharam64} (see~\cite[Theorem~1.1.7]{Aaronson}).

By~\cite[Lemma~3.1]{EMV21}, $f:X\to X$ is modelled by an aperiodic (that is, topologically mixing) Young tower $(\Delta,\mu_\Delta,f_\Delta)$.
Recall that $\Delta=\{(z,\ell):z\in Z,\,0\le \ell\le R(z)-1\}$ where $R:Z\to\Z^+$ is a return time (not necessarily the first return under $f$) to a subset $Z\subset Y$ with $0<\mu(Z)<\infty$.
Moreover, by the proof of~\cite[Lemma~3.1]{EMV21}, $\mu_{\Delta}(\{(z,0):R(z)=1\})>0$.
Under these conditions, we
give a simple argument for exactness of $f_\Delta$, and hence $f$, following~\cite{Lenci16,Young98}.
(We note that~\cite[Lemma~5]{Young98} already establishes exactness of $f_\Delta$ in the finite measure case.)

By definition of a (one-sided) Young tower, the return map $f_\Delta^R:Z\to Z$ is a full-branch Gibbs-Markov map. Identifying $Z$ with the base of the tower, it follows from our assumption on $R$ that $Z\subset f_\Delta Z$. Using bounded distortion of  $f_\Delta^R$ it follows that there exists $\eps>0$ such that
$\mu_\Delta(E\cap f_\Delta E)>0$ for any $E\subset Z$ with $\mu_\Delta(E)>(1-\eps) \mu_\Delta(Z)$.

Let $A\subset\Delta$ with $\mu_\Delta(A)>0$.
By the proof of~\cite[Lemma~5]{Young98}, there exists $n\ge1$ such that
$\mu_\Delta(Z\cap f_\Delta^nA)>(1-\eps) \mu_\Delta(Z)$. Hence
\[
\mu_\Delta(f_\Delta^nA\cap f_\Delta^{n+1} A) \ge
\mu_\Delta(Z\cap f_\Delta^nA\cap f_\Delta(Z\cap f_\Delta^nA) )>0.
\]
Exactness follows by~\cite[Lemma~2.1]{MiernowskiNogueira13}.
\end{proof}

The abstract setup from Section~\ref{sec:main} has to be generalised slightly: we can take $q=1$ but the value of $|\bbA|$ now depends on $j$ since the second coordinate of $f^j$ is $4^j$-to-one.
Specifically, for each $j\ge1$, $n\ge n_0$, we divide 
$Y_{j+n}$ into subsets 
\[
 Y_{j,n,r}=\{(y,\theta)\in Y_{j+n}: (r-1) 4^{-j}\le \theta \le r^{-j}\}, \quad
 1\le r\le q_j=4^j,
\]
so that $f^j$ restricts to a bijection from $Y_{j,n,r}$ onto $X_n$.
Sums of the form $\sum_{\psi(r)=p}$ are replaced by $\sum_{r=1}^{4^j}$.
In particular,~\eqref{eq:J} becomes
\begin{equation} \label{eq:tildeEMV}
\lim_{j\to\infty}j^\alpha\sum_{n=n_0}^\infty\sum_{r=1}^{4^j} \int_{Y_{j,n,r}}g\circ f^j\,d\mu=0
\end{equation}
for all $g\in L^\infty$ satisfying $\int_{X_n}g\,d\mu=0$ for all $n\ge n_0$.
The counterpart of~\eqref{eq:J} is
that there are constants
$c_{j,n}>0$ such that
\begin{equation} \label{eq:JEMV}
\sum_{n=n_0}^\infty n^{-\alpha}\sup_{X_n}|J_{j,n}-c_{j,n}|=o(j^{-\alpha})
\quad\text{as $j\to\infty$}
\end{equation}
where
\[
J_{j,n}(x)=\frac{\sum_{r=1}^{4^j}h(y_{j,n,r})(\det Df^j(y_{j,n,r}))^{-1}}
{\sum_{\ell=1}^\infty\sum_{r=1}^{4^\ell} h(y_{\ell,n,r})(\det Df^\ell(y_{\ell,n,r}))^{-1}}
\]
for $x\in X_n$ and $y_{\ell,n,r}\in Y_{\ell,n,r}$ with $f^jy_{\ell,n,r}=x$.

It remains to verify~\eqref{eq:JEMV}.
By~\cite[Lemma~2.9]{EMV21}, incorporating the strengthened expansion for $u(x,\theta)$, we find that
the diffeomorphism $f^j:Y_{j,n,r}\to X_n$ satisfies
\[
        \{\det Df^j\}^{-1}= 4^{-(j+1)} \Big(\frac{n}{j+n}\Big)^{\alpha+1}
\big(1+O(n^{-\kappa'})\big)
\]
where the implied constant is independent of $j$, $r$ and $n$ uniformly on $Y$.
(There is a factor of $4^j$ arising from the second coordinate of  $f^j$ and a factor of $4$ from the first iterate using the second branch $x\mapsto 4x-3$.)

By the construction in~\cite{EMV21}, $\lim_{n\to\infty}\dist(y_{j,n,r},\{x=\tfrac34\})=0$ uniformly in $j$ and $r$. Moreover
the one-sided limit $h(\tfrac34^+,\theta)$ exists for a.e.\ $\theta\in\T$
by~\cite[Proposition~4.11]{EMV21} and
\[
 \lim_{n\to\infty}\sum_{r=1}^{4^j}h(y_{j,n,r})4^{-j} = \int_\T h(\tfrac34^+,\theta)\,d\theta 
\]
uniformly in $j$.
Once again we find that $J_{j,n}\sim\alpha n^\alpha(j+n)^{-(\alpha+1)}$ as $n\to\infty$ uniformly in $j$ on $X_n$, verifying condition~\eqref{eq:J}.
\end{example}

\subsection{Parabolic rational maps of the complex plane}
\label{sec:parabolic}

\begin{example}
\label{ex:parabolic}
Let $\bC$ denote the Riemann sphere with the spherical metric.
We consider rational maps $f:\bC\to\bC$ with $d\ge1$ rationally indifferent periodic orbits $\xi_1,\dots,\xi_d$.
For simplicity, we suppose as usual that these periodic orbits are equally indifferent (see~\eqref{eq:parabolic} below).
A standard reference is~\cite{ADU93}.

Let $X=J$ denote the Julia set of $f$, with Hausdorff dimension $\cH(J)$.
By~\cite[Proof of Theorem~8.7]{ADU93}, there is a unique $\cH(J)$-conformal probability measure $m$ supported on $J$.
This means that $m(fA)=\int_A |f'|^{\cH(J)}dm$ for sets $A\subset J$ on which $A$ is injective.
By~\cite[Theorem~9.6]{ADU93}, there is a unique (up to scaling) $\sigma$-finite absolutely continuous $f$-invariant measure $\mu$ on $J$, and $\mu$ is equivalent to $m$. 
Moreover, $f:J\to J$ is conservative and exact.

Necessary and sufficient conditions for $\mu(J)=\infty$ are given in~\cite{ADU93}, see below.
For such maps, it is natural to define $g\in L^\infty$ to be a global observable if
there exists $\bar g\in\R$ such that
\[
\lim_{\eps\to0^+}\frac{\int_{J\setminus\bigcup_{k=1}^d B_\eps(\xi_k)}g\,d\mu}{
\mu({J\setminus\bigcup_{k=1}^d B_\eps(\xi_k)})}=\bar g.
\]

We prove:
\begin{thm} \label{thm:ADU}
Suppose that $f:\bC\to\bC$ is a parabolic rational map with $\mu(J)=\infty$.
Then $f$ is global-local mixing.
\end{thm}

Applying Proposition~\ref{prop:power}, we can suppose without loss of generality that all the indifferent periodic points of $f$ are equally indifferent fixed points $\xi_1,\dots,\xi_d$.
Then 
\begin{equation} \label{eq:parabolic}
f(z)=z+b_k(z-\xi_k)^{1+\beta}+O((z-\xi_k)^{2+\beta})
\quad\text{as $z\to\xi_k$, $\;k=1,\dots,d$,}
\end{equation} 
where $\beta\ge1$ is a positive integer and $b_k\neq0$, $k=1,\dots,d$.
By~\cite[Corollary~8.6 and Theorem~8.8]{ADU93}, $J$ has Hausdorff dimension 
$\cH(J)\in(\frac{\beta}{\beta+1},2)$.
By~\cite[Theorem~9.8]{ADU93}, $\mu(J)=\infty$ if and only if 
$\cH(J)\in(\frac{\beta}{\beta+1}, \frac{2\beta}{\beta+1}]$.
Focusing on the infinite measure case, we set
$\alpha=(1+1/\beta)\cH(J)-1\in(0,1]$.\footnote{Our $\alpha$ corresponds to $\alpha-1$ in~\cite{ADU93}.}

Choose $\eps>0$ so that $B_\eps(\xi_i)\cap fB_\eps(\xi_j)\neq\emptyset$ for all $i\neq j$.
There exists a finite measurable Markov partition $\cA$ of $\bC$ consisting of sets $A$ with $m(A)>0$ and $\diam A<\eps/3$ and such that 
the neutral fixed points $\xi_j$ lie in the interior of partition elements.
(See for example~\cite[Section~9]{ADU93}.)
This yields a finite Markov partition $\cP=\{A\cap J: A\in\cA\}$ for $f:J\to J$.

We define the inducing set $Y$ to be the union of all the partition elements $I\in\cP$ that intersect
$J\setminus\bigcup_{k=1}^d B_\eps(\xi_k)$.
By construction, 
\[
J\cap \bigcup_{k=1}^d B_{\eps/3}(\xi_k)\subset J\setminus Y
\subset\bigcup_{k=1}^d B_\eps(\xi_k).
\]

We can now define the first hit time $\tau:J\to\Z^+$ and the subsets
$Y_n\subset Y$ and $X_n\subset J\setminus Y$ as usual.
Take $q=d$ and set
$X_{n,p}=X_n\cap B_{\eps/3}(\xi_p)$ for $1\le p\le q$.
Write 
  $\bigcup_{\psi(r)=p}Y_{n,r}=Y_n\cap f^{-1}X_{n+1,p}$
and note that the number of preimages $|\psi^{-1}(p)|$ is bounded by $|\cP|-1$ for each $p$.
Hence we can take $\bbA$ to be a finite indexing set with $|\bbA|=\sum_{p=1}^q|\psi^{-1}(p)|$. 

It follows from the choice of $\eps$ that $X_{n,p}\subset B_{\eps/3}(\xi_p)$ for $n\ge1$, $p=1,\dots,q$, and that $f$ restricts to diffeomorphisms from
$X_{n+1,p}$ to $X_{n,p}$ and from
$Y_{n+1,r}$ to $X_{n,p}$ for all $r\in\psi^{-1}(p)$.
In particular the first return map $F=f^\tau:Y\to Y$ is a Markov map with finitely many images (namely $\{fX_{1,p}: 1\le p\le q\}\cup\{fI\in\cP:\tau|_I=1\}$).

The expansion~\eqref{eq:parabolic} holds on $\bigcup_{n\ge1}X_{n,p}$ for each $p$.
By~\cite[Proposition~8.3]{ADU93}, $\dist(X_{n,p},\xi_p)=O(n^{-1/\beta})$ for each $p$.
For $r\in\psi^{-1}(p)$, there is a unique $\eta_r$ such that $f\eta_r=\xi_p$, and
$\dist(Y_{n,r},\eta_r)=O(n^{-1/\beta})$.
Also, $F$ satisfies the usual bounded distortion estimates with respect to the conformal measure $m$ and hence $\mu$ (see~\cite{DenkerUrbanski91} for such estimates). 
Hence $F:Y\to Y$ is a Gibbs-Markov map. 

Repulsive (hyperbolic) periodic points are dense in $J$ so there is a repulsive periodic orbit lying in the union of the interiors of finitely many partition elements. By Proposition~\ref{prop:power}, we can suppose without loss of generality that this is a fixed point $x_0$. Now shrink $\eps$ so that $x_0\in Y$. Then $F$ is an aperiodic Gibbs-Markov map and hence is mixing.

The density $h=d\mu/dm$ is continuous on $Y$ with respect to the underlying metric (this follows from calculations in~\cite[Section~9]{ADU93}).

\begin{prop} \label{prop:m}
Conditions~\eqref{eq:Y} and~\eqref{eq:K} are satisfied.
\end{prop}

\begin{proof}
By~\cite[Lemma~8.9]{ADU93},
\[
m(X_{n+1,p})=m(f_p^{-n}X_{1,p})\sim m(X_{1,p})n^{-(\alpha+1)}.
\]
Let $r\in\bbA$ with $\psi(r)=p$; 
then $fY_{n,r}=X_{n-1,p}$.
Since $\dist (\eta_r,Y_{n,r})\to0$ as $n\to\infty$, it follows 
 from conformality of $m$ and continuity of $f'$ that
\[
m(X_{n-1,p})  =m(fY_{n,r})=\int_{Y_{n,r}}|f'|^{\cH(J)}dm
 \sim |f'(\eta_r)|^{\cH(J)} m(Y_{n,r}).
\]
Since the density $h$ is continuous,
\begin{equation} \label{eq:muY}
\mu(Y_{n,r})  =\int_{Y_{n,r}}h\,dm
 \sim h(\eta_r)m(Y_{n,r})\sim h(\eta_r) |f'(\eta_r)|^{-\cH(J)}m(X_{1,p}) n^{-(\alpha+1)}.
\end{equation} 
Hence condition~\eqref{eq:Y} is satisfied with
$\gamma_p=\alpha^{-1}m(X_{1,p})\sum_{\psi(r)=p}h(\eta_r) |f'(\eta_r)|^{-\cH(J)}$.

Condition~\eqref{eq:K} follows from~\eqref{eq:muY} and the fact that $F$ is a mixing Gibbs-Markov map by~\cite{Gouezel11,MT12}.
\end{proof}

\begin{prop} Condition~\eqref{eq:J} is satisfied.
\end{prop}

\begin{proof}
We verify the conditions in Proposition~\ref{prop:JJ}.
We already saw that condition~(a) is satisfied.

Fix $p=1,\dots,q$ and $r\in\bbA$ with $\psi(r)=p$.
For notational simplicity, suppose that $\xi_p=0$.
 Let $f_0$ be the branch of $f$ near $0$. 
For $z$ close enough to $0$, let $z_n=f_0^{-n}z$.
By~\cite[Proof of Theorem~8.4]{ADU93},
\[
t(z)^{-1}|z_n-z_{n+1}|\le |(f_0^{-n})'(z)|\,|z-z_1|\le t(z)|z_n-z_{n+1}|
\]
where $\lim_{z\to0}t(z)=1$.
Let $y\in Y_{n+\ell}$ with $u=f^\ell y\in X_n$.
Then 
\[
t(u)^{-1}\ \le \frac{|u_{\ell-1}-u_{\ell}|}{|u-u_1|}
|(f_0^{-(\ell-1)})'(u)|^{-1}\le t(u) .
\]

Now, 
\[
(f^\ell)'(y)=f'(y)\cdot (f_0^{\ell-1})'(fy)
=f'(y)\cdot  \{(f_0^{-(\ell-1)})'(u)\}^{-1}.
\]
Hence,
\[
t(u)^{-1} |f'(y)| \le \frac{|u_{\ell-1}-u_{\ell}|}{|u-u_1|} |(f^\ell)'(y)| \le t(u) |f'(y)| .
\]
Write $u=z_{n-1}$ where $z=f_0^{n-1}u\in X_1$. Then
\[
\frac{|u_{\ell-1}-u_{\ell}|}{|u-u_1|}= 
\frac{|z_{n+\ell-2}-z_{n+\ell-1}|}{|z_{n-1}-z_n|}
\to \Big(\frac{n}{n+\ell}\Bigr)^{1+1/\beta}
\]
as $n\to\infty$ uniformly in $\ell$ and uniformly in $z\in X_{1,p}$
 by~\cite[Proposition~8.3]{ADU93}.
This verifies conditions~(b) and~(c) of Proposition~\ref{prop:JJ} with
$B=1/\beta$ and $\omega_r=|f'(\tilde\xi_r)|$.
\end{proof}

\end{example}

 \paragraph{Acknowledgements}
DC was partially supported by the São Paulo Research Foundation (FAPESP) grant 2022/16259-2, and the Fundação Carlos Chagas Filho de Amparo à Pesquisa do Estado do Rio de Janeiro (FAPERJ) grant E-26/200.027/2025.

IM was partially supported by the São Paulo Research Foundation (FAPESP) grant 
2024/22093-5.

IM is grateful to IMPA and UFRJ for their hospitality during visits to Rio de Janeiro where much of this research was carried out.

\end{document}